\documentclass[11pt]{amsart}

\usepackage{amscd,amsmath,amssymb,amsfonts,amsthm,mathrsfs}
\usepackage[a4paper,text={160mm,240mm},centering,headsep=5mm,footskip=10mm]{geometry}
\usepackage[all]{xy}

\usepackage[colorlinks,linkcolor=black,anchorcolor=black,citecolor=black]{hyperref}

\usepackage[square,sort,comma,numbers]{natbib}
\usepackage{stmaryrd}
\usepackage{array}
\usepackage{color}

\newtheorem{theorem}{Theorem}[subsection]
\newtheorem{proposition}[theorem]{Proposition}
\newtheorem{lemma}[theorem]{Lemma}
\newtheorem{remark}[theorem]{Remark}
\newtheorem{corollary}[theorem]{Corollary}

\newtheorem{definition}[theorem]{Definition}

\newtheorem{claim}[theorem]{Claim}

\numberwithin{equation}{subsection}
\newtheorem{theoremintr}{Theorem}

\newcommand{\fm}{\mathfrak{m}}

\def\FnuR#1#2{G^{#1}\nu_R(#2)}
\def\eb{\underline{n}}
\def\ebp{{[\underline{n}/p]}}
\def\ul#1{\underline{#1}}
\def\Dp{{[D/p]}}

\def\dlog#1{\displaystyle{\frac{d#1}{#1}}}
\def\FKMRr#1{U^{#1}K^M_r(R)}
\def\FkMRr#1#2{U^{#1}k^M_r(R)_{#2}}
\def\grikMRr#1#2{gr^{#1,i}k^M_r(R)_{#2}}

\def\KMRr{K^M_r(R)}

\def\DRW#1#2#3{W_{#1}\Omega^{#2}_{#3}}
\def\DRWlog#1#2#3{W_{#1}\Omega^{#2}_{#3,\log}}
\def\DR#1#2{\Omega^{#1}_{#2}}
\def\DRlog#1#2{\Omega^{#1}_{#2,\log}}
\def\WXDlog#1#2{W_{#1}\Omega^{#2}_{X|D,\log}}
\def\WXDplog#1#2{W_{#1}\Omega^{#2}_{X|\Dp,\log}}

\def\FDR#1#2#3{G^{#1}\Omega^{#2}_{#3}}

\def\Cni{C^{-1}_{\eb,i}}
\def\wni#1{\omega^{#1}_{\eb,i}}
\def\wrni#1{\omega^{r-1}_{\eb,i}}
\def\wnpi#1{\omega^{#1}_{\ebp,i}}

\def\Zni#1{Z^{#1}_{\eb,i}}

\def\Bni#1{B^{#1}_{\eb,i}}

\def\ZZni#1#2{Z^{#1}_{#2|\eb,i}}

\def\ZZnpi#1#2{Z^{#1}_{#2|\ebp,i}}

\def\BBni#1#2{B^{#1}_{#2|\eb,i}}

\def\BBnpi#1#2{B^{#1}_{#2|\ebp,i}}

\def\Spec{\mathrm{Spec}}
\def\Supp{\mathrm{Supp}}

\def\Hom{\mathrm{Hom}}

\def\sHom{{\mathcal H}om}  
\def\cL{\mathscr{L}}
\def\cX{\mathscr{X}}
\def\cS{\mathscr{S}}
\def\cO{\mathcal{O}}

\def\bN{\mathbb{N}}

\def\qwith{\;\text{with}\;}
\def\Image{\mathrm{Image}}
\def\rmapo #1{\xrightarrow{#1}}
\def\rmapou #1#2{\xrightarrow[#1]{#2}}
\def\Ker{\mathrm{Ker}}
\def\isom{\xrightarrow{\cong}}

\def\Fp{\mathbb{F}_p}
\def\Z{\mathbb{Z}}
\def\FnuR#1#2{G^{#1}\nu_R(#2)}
\def\eb{{\underline n}}
\def\dlog#1{\displaystyle{\frac{d#1}{#1}}}
\def\D#1{D_{\underline{#1}}}
\def\nlam{(n_{\lambda})_{\lambda\in \Lambda}}
\def\Pf{\mathit{Pf}}

\def\RDRWlog#1#2#3{W_{#1}\Omega^{#2}_{X|D_{#3},\log}}
\def\FDRW#1#2#3{W_{#1}\Omega^{#2}_{X|D_{#3}}}

\def\cO{\mathcal{O}}

\def\Fil{\mathrm{Fil}}
\def\un {\underline n}

\def\Q{\mathbb Q}

\begin{document}

	\title[Duality for relative logarithmic de Rham-Witt sheaves]
	{Duality for relative logarithmic de Rham-Witt sheaves and wildly ramified class field theory over finite fields }
	\author{ Uwe Jannsen}
	\address{Fakult\"at f\"ur Mathematik, Universit\"at Regensburg, 93040 Regensburg, Germany}
	\email{uwe.jannsen@mathematik.uni-regensburg.de}
	
	\author{Shuji Saito}
	\address{Interactive Research Center of Science,
		Graduate School of Science and Engineering,
		Tokyo Institute of Technology,
		2-12-1 Okayama, Meguro,
		Tokyo 152-8551,
		Japan}
	\email{sshuji@msb.biglobe.ne.jp}
	
	\author{Yigeng Zhao}
	\address{Fakult\"at f\"ur Mathematik, Universit\"at Regensburg, 93040 Regensburg, Germany}
	\email{yigeng.zhao@mathematik.uni-regensburg.de}
	
	\thanks{The authors are supported by the DFG through CRC 1085 \emph{Higher Invariants} (Universit\"at Regensburg)}

	\begin{abstract}
		In order to study $p$-adic \'etale cohomology of an open subvariety $U$ of a smooth proper variety $X$ over a perfect field of characteristic $p>0$, we introduce new $p$-primary torsion sheaves. It is a modification of the logarithmic de Rham-Witt sheaves of $X$ depending on effective divisors $D$ supported in $X-U$. Then we establish a perfect duality between cohomology groups of the logarithmic de Rham-Witt cohomology of $U$ and an inverse limit of those of the mentioned modified sheaves. Over a finite field, the duality can be used to study wild ramification class field theory for the open subvariety $U$.
	\end{abstract}
	\date{}
	\maketitle
	\tableofcontents

	
	\setcounter{page}{1}
	\setcounter{section}{0}
	\pagenumbering{arabic}
	
	\section*{Introduction}
	\bigskip
	
	Let $k$ be a perfect field of characteristic $p>0$ and let $X$ be a smooth proper variety of dimension $d$ over $k$. The logarithmic de Rham-Witt sheaves $\DRWlog{m}{r}{X}$ are defined as the subsheaves of the de Rham-Witt sheaves $\DRW{m}{r}{X}$, 
	which are \'etale locally generated by sections $d\log[x_1]_m\wedge \ldots\wedge d\log[x_r]_m $ with $x_\nu \in \cO_X^{\times}$ for all $\nu$ (\cite{illusiederham}). 
	By the Gersten resolution (\cite{rost1996chow},\cite{kerzmilnorfinite},\cite{grossuwa}) and the Bloch-Gabber-Kato theorem (\cite{blochkato}), the $d\log$ map induces an isomorphism of \'etale sheaves
	\begin{align*}\label{blochgabberkato}
	d\log[-]: \mathcal{K}^M_{r,X}/p^m &\xrightarrow{\cong}\DRWlog{m}{r}{X}\;;\;
	\{x_1,\dots,x_r\} \mapsto d\log [x_1]_m \wedge \cdots \wedge d\log [x_r]_m, \tag{1}
	\end{align*}
	where $\mathcal{K}^M_{r,X}$ is the sheaf of Milnor $K$-groups. 
	It is conceived as a $p$-adic analogue of the $\ell$-adic sheaf $\mu_{\ell^m}^{\otimes r}$ with $\ell\neq p$. If $k$ is a finite field, 
	there is a non-degenerated pairing of finite groups due to Milne (\cite{milneduality}):
	$$
	H^i(X,W_m\Omega^r_{X,\log})\times H^{d+1-i}(X,W_m\Omega^{d-r}_{X,\log})\rightarrow H^{d+1}(X,W_m\Omega^d_{X,\log}))\xrightarrow{\text{Tr}} \mathbb Z/p^m\mathbb Z\, .
	$$
	It induces a natural isomorphism
	$$
	H^d(X,W_m\Omega^d_{X,\log})  \cong    H^1(X,\mathbb Z/p^m\mathbb Z)^\vee \cong \pi^{ab}_1(X)/p^m
	$$
	where $A^\vee$ is the Pontryagin dual of a discrete abelian group and $\pi^{ab}_1(X)$ is the maximal abelian quotient of Grothendieck's \'etale fundamental group of $X$. 
	This gives a description of $\pi^{ab}_1(X)/p^m$ in terms of \'etale cohomology with $p$-adic coefficient.
	For $\ell$-adic \'etale cohomology, we also have a non-degenerated pairing of finite groups for a smooth non-proper variety $U$ of dimension $d$ over a finite field $k$ 
	(\cite{SGA41/2} and \cite{saito1989}) 
	$$
	H^i(U,\mathbb Z/\ell^m(j))\times H_c^{2d+1-i}(U,\mathbb Z/\ell^m(d-j))\rightarrow H_c^{2d+1}(U,\mathbb Z/\ell^m(d))\cong \mathbb Z/\ell^m\mathbb Z,
	$$
	which can be used to describe $\pi_1^{ab}(U)/\ell^m$ by $\ell$-adic \'etale cohomology:
	$$
	H^{2d}_c(U,\mathbb Z/\ell^m(d)) \cong H^1(U,\mathbb Z/\ell^m)^\vee \cong  \pi^{ab}_1(U)/\ell^m.
	$$
	In the $p$-adic setting there is no obvious analogue of \'etale cohomology with compact support for logarithmic de Rham-Witt sheaves. 
	\medbreak
	
	In this paper, we propose a new approach. Let $X$ be a proper smooth variety over a
	perfect field $k$ as before, and let $j: U \hookrightarrow X$ be the complement of an effective divisor $D$ such that Supp($D$) has simple normal crossings.  We introduce new $p$-primary torsion sheaves $\DRWlog m r {X|D} $ (see Definition \ref{relative.definition}), which we call \emph{relative logarithmic de Rham-Witt sheaves}. 
	It is defined as the subsheaf of the de Rham-Witt sheaf $\DRW{m}{r}{X}$ 
	which is \'etale locally generated by sections $d\log[x_1]_m\wedge \ldots\wedge d\log[x_r]_m $ with $x_1\in \Ker(\cO_X^{\times} \to \cO_D^{\times})$, and $x_\nu \in j_*\cO_U^{\times}$ for all $\nu$. As in the classical situation, we have the following theorem:
	\begin{theoremintr}[see Theorem \ref{relative.blochkato}]
		The map $d\log$ induces an isomorphism
		\begin{align*}\label{relativeblochgabberkato}
		d\log[-]: \mathcal{K}^M_{r,X|D}/(p^m\mathcal{K}^M_{r,X}\cap \mathcal{K}^M_{r,X|D} ) &\xrightarrow{\cong}\DRWlog{m}{r}{X|D}\;;\; \tag{2}\\
		\{x_1,\dots,x_r\} &\mapsto d\log [x_1]_m \wedge \cdots \wedge d\log [x_r]_m. 
		\end{align*}
	\end{theoremintr}
	Here $\mathcal{K}^M_{r,X|D}$ is the sheaf of relative Milnor $K$-groups which has been studied by one of the authors (S. Saito) and K. R\"ulling in \cite{ruellingsaito}. 
	
	If $D_1\geq D_2$, we have inclusions (see Proposition \ref{transition.injective})
	\begin{equation*}\label{inclusion}
	\DRWlog m r {X|D_1} \subseteq\DRWlog m r {X|D_2} \subseteq \DRWlog m r X,
	\end{equation*}  
	and thus obtain a pro-system of $\Z/p^m\Z$-sheaves $\lq\lq\varprojlim\limits_D"\DRWlog{m}{r}{X|D}$, where $D$ runs over the set of all effective divisors with $\text{Supp}(D)\subset X-U$. 
	
	In case $m=1$ these sheaves are related to sheaves of differential forms by the exact sequence (see Theorem \ref{logtocoherent}):
	\begin{equation}\label{log.to.coh}
	0 \to \Omega_{X|D,\log}^r \to \Omega_{X|D}^r \xrightarrow{1-C^{-1}} \Omega_{X|D}^r/d\Omega_{X|D}^{r-1} \to 0,\tag{3}\end{equation}
	where $\Omega_{X|D}^r=\Omega_{X}^r(\log D)\otimes_{\cO_X}\cO_X(-D)$ and $C^{-1}$ is the inverse Cartier morphism. 
	In order to extend the above exact sequence to the case $m>1$,
	we need introduce \emph{the filtered relative de Rham-Witt complex $W_m\Omega_{X|D}^\bullet$} for which we have $W_1\Omega_{X|D}^\bullet=\Omega_{X|D}^\bullet$
	(see \S \ref{defn.filterdRW} and Theorem \ref{exactproobj}). Its construction uses the de Rham-Witt complexes in log geometry \cite{hyodokato}, which can be seen as the higher analogue of $\Omega_{X}^r(\log D)$.

	Using the generalization of \eqref{log.to.coh} to the case $m>1$,
	we can define a pairing between $\DRWlog{m}{r}{U}$ and the pro-system $\lq\lq \varprojlim\limits_D"\DRWlog{m}{d-r}{X|D}$  and obtain the following theorem:
	
	\begin{theoremintr}[see Theorem \ref{duality.thm.finite}]
		Let $X$, $D$ and $U$ be as above and assume that $k$ is finite.
		Then $H^{j}(X, \DRWlog m {r} {X|D})$ are finite and there are natural perfect pairings of topological abelian groups
		\[ H^i(U, \DRWlog m r U) \times \varprojlim_D H^{d+1-i}(X, \DRWlog m {d-r} {X|D}) \to H^{d+1}(X, \DRWlog m d X) \xrightarrow{\text{Tr}} \Z/p^m\Z,\]
		where the first group is endowed with discrete topology, and the second is endowed with profinite topology.
	\end{theoremintr}
	
	From the case $i=1$ and $r=0$ of the above theorem, we get a natural isomorphism
	$$
	\varprojlim\limits_{D} H^d(X,\DRWlog m d {X|D})\xrightarrow{\cong} H^1(U,\Z/p^m\Z)^\vee\cong \pi^{\text{ab}}_1(U)/p^m,
	$$
	which gives rise a series of quotients $\pi^{\text{ab}}_1(X,D)/p^m$ of $\pi^{\text{ab}}_1(U)/p^m$ using the inverse limit. It is thought of as classifying abelian \'etale covering of $U$ whose degree divides $p^m$ and ramification is bounded by the divisor $D$. 
	
	
	\medbreak
	
	One of the authors (Y. Zhao) \cite{zhaoduality} has proved a similar duality theorem for a projective semi-stable scheme over an equi-characteristic discrete valuation ring $k[[t]]$ with $k$ finite.
	\medbreak

	When the base field $k$ is prefect but not necessarily finite, we follow the method of Milne \cite{milneduality} and work in the category $\cS(p^m)$ of $\Z/p^m\Z$-sheaves on perfect \'etale site $(\Pf/k)_{\text{\'et}}$ (see \S \ref{relative.perfect.site}). 
	Let $D^b(\cS(p^m))$ be the derived category of bounded complexes in $\cS(p^m)$. We then get from $X,\;D$ objects of $D^b(\cS(p^m))$:
	\[R\pi_*\DRWlog m {d-r} {X|D}\;\; \text{ and }\;\; 
	R\pi_*Rj_*\DRWlog m r U,\]
	where $\pi:X\to S=\Spec(k)$ is the structure morphism and $j:U\to X$ is the open immersion. Then our duality theorem reads:
	
	\begin{theoremintr}[see Theorem \ref{duality.thm.perfect}]
		There are a natural isomorphism in $D^b(\cS(p^m))$:
		\[	 R\varprojlim\limits_{D} R\pi_*\DRWlog m {d-r} {X|D} \xrightarrow{\cong} R\sHom_{D^b(\cS(p^m))}( R\pi_*Rj_*\DRWlog m r U , \Z/p^m\Z)[-d],\]
		where $R\varprojlim\limits_{D}$ denotes the homotopy limit over effective Cartier divisors supported on $X-U$.
	\end{theoremintr}

	\bigskip

	The paper is organized as follows.

	In \S 1, we study the two important results on the relative logarithmic de Rham-Witt sheaves: the first one is a computation of the kernel of the restriction map  $R^{m-1}: \DRWlog{m}{r}{X|D} \to \DRWlog{1}{r}{X|D}$; the second is the exact sequence (\ref{log.to.coh}).
	
	In order to define the desired pairing, we introduce the filtered de Rham-Witt complexes in \S2, and study the behavior of Frobenius and Verschiebung on these complexes. 
	
	Using two-term complexes, we define the pairing in \S 3 and prove its perfectness when the base field $k$ is finite in \S 4. The last section \S 5 is on the duality over a general perfect field.

	\bigskip
	\textbf{Acknowledgments}
	The authors thank Moritz Kerz for his advice, especially on the construction of a pairing.  The authors would like to thank the anonymous referee for his/her numerous valuable comments and suggestions to improve the quality of this paper. 
	In particular, his/her suggestion on a simplification of the construction of the filtered de Rham-Witt complexes in \S 2 is invaluable.
	\bigskip
	
	\section{Relative logarithmic de Rham-Witt sheaves}
	\bigskip
	
	Let $X$ be a smooth proper variety of dimension $d$ over a perfect field $k$ of characteristic $p>0$,  let $D$ be an effective divisor such that $\Supp(D)$ is a simple normal crossing divisor on $X$, and let $j: U:=X-D \hookrightarrow X$ be the complement of $D$. 
	\subsection{Basic properties}
	\begin{definition}\label{relative.definition}
		For $r\in \bN$ let
		$$
		W_m\Omega^r_{X|D,\log}\subset j_\ast W_m \Omega^r_{U,\log}
		$$
		be the subsheaf generated \'{e}tale locally by sections
		$$
		d\log[x_1]_m\wedge \ldots\wedge d\log[x_r]_m \ \text{with} \ \ x_1 \in \Ker(\cO_X^{\times} \to \cO_D^{\times}) , x_{\nu} \in j_*\mathcal{O}_U^{\times} \ \text{for all}\  \nu.
		$$
	\end{definition}

	For $r\in \mathbb{N}$ let $\mathcal{K}^M_{r,X}$ be the $r$-th Milnor $K$-sheaf on  $X_{\text{\'et}}$ given by 
	\[ V \mapsto \Ker(\bigoplus_{\eta\in V^{(0)}}K^M_r(k(\eta))\xrightarrow{\oplus\partial_x}\bigoplus_{x\in V^{(1)}} K^M_{r-1}(k(x)))  \;\;\text{ for an \'etale $V\to X$,}    \]
	where $V^{(i)}$ is the set of points of codimension $i$ in $V$, for $i=0,1$, and $\partial_x: K^M_r(k(\eta)) \to K^M_r(k(x))$ is the tame symbol from \cite[\S 4]{basstate}. 
	By \cite[Prop.10 (8) and Thm.13]{kerzmilnorfinite}, $\mathcal{K}^M_{r,X}$ is \'etale locally generated by symbols 
	$\{x_1,\cdots, x_r \}$ with $x_i\in \cO_{X,x}^{\times}$. 
	We have a natural isomorphism of \'etale sheaves
	\begin{align}\label{dlogiso}
	d\log[-]: \mathcal{K}^M_{r,X}/p^m &\xrightarrow{\cong} W_m\Omega^r_{X,\log}\\
	\{x_1,\dots,x_r\} &\mapsto d\log [x_1]_m \wedge \cdots \wedge d\log [x_r]_m. \notag
	\end{align}
	This follows from the Gersten resolutions of $\epsilon_*\mathcal{K}^M_{r,X}$
	and $\epsilon_*W_m\Omega^r_{X,\log}$ from \cite{rost1996chow} and \cite{grossuwa}
	together with the Bloch-Gabber-Kato theorem\cite{blochkato}, where 
	$\epsilon: X_{\text{\'et}}\to X_{\text{Zar}}$ is the map of sites.
	
	\begin{definition}(\cite[Def. 2.4]{ruellingsaito})
		For $r \in \mathbb{N}$, we define the relative Milnor $K$-sheaf 
		$\mathcal{K}^M_{r, X|D}$ to be image of the map
		\[ \Ker(\cO_X^{\times} \to \cO_D^{\times}) \otimes_{\Z} j_*\mathcal{K}^M_{r-1, U} \to  j_*\mathcal{K}^M_{r, U};\ x\otimes \{x_1, \cdots, x_{r-1} \} \mapsto \{x, x_1,\cdots, x_r \}.     \]
	\end{definition}

	Using some symbol calculations, we get the following proposition: 
	\begin{proposition}(\cite[Cor. 2.9]{ruellingsaito})\label{milnor.transition}
		Let $D_1,D_2$ be two effective divisors on $X$ whose supports are simple normal crossing divisors. Assume $D_1 \leq D_2$. Then we have the inclusions of sheaves
		\[ \mathcal{K}^M_{r,X|D_2}\subset \mathcal{K}^M_{r,X|D_1}\subset \mathcal{K}^M_{r,X}.\]
	\end{proposition}
	
	\begin{corollary}\label{transition.injective}
		Under the assumption of Proposition \ref{milnor.transition},  we have inclusions
		\[ W_m\Omega^r_{X|D_2,\log}\subset W_m\Omega^r_{X|D_1,\log}\subset W_m\Omega^r_{X,\log}.\]
	\end{corollary}
	\begin{proof}
		This follows from the fact that the sheaf $W_m\Omega^r_{X|D,\log}$ is the image of $\mathcal{K}^M_{r, X|D}$ under the map $d\log[-]$.
	\end{proof}
	
	The isomorphism (\ref{dlogiso}) also has the following relative version.
	\begin{theorem}\label{relative.blochkato}
		The $d\log$ map induces an isomorphism of \'etale sheaves
		\begin{align}\label{relativedlog}
		d\log[-]: \mathcal{K}^M_{r,X|D}/(p^m\mathcal{K}^M_{r,X}\cap \mathcal{K}^M_{r,X|D} ) &\xrightarrow{\cong} W_m\Omega^r_{X|D,\log}\\
		\{x_1,\dots,x_r\} &\mapsto d\log [x_1]_m \wedge \cdots \wedge d\log [x_r]_m. \notag
		\end{align}
	\end{theorem}
	\begin{proof}
		The assertion follows directly by the following commutative diagram 
		\[ \xymatrix{ 
			\mathcal{K}^M_{r,X|D}/(p^m\mathcal{K}^M_{r,X}\cap \mathcal{K}^M_{r,X|D} ) \ar@{^(->}[r] \ar@{->>}[d]_{d\log}  &\mathcal{K}^M_{r,X}/p^m \ar[d]^{\cong}_{d\log}\\
			\DRWlog{m}{r}{X|D} \ar@{^(->}[r] & 	\DRWlog{m}{r}{X}.
		}\]
	\end{proof}

	In the rest of this section, we will prove two fundamental results for the relative logarithmic de Rham-Witt sheaves. 
	
	\begin{theorem}\label{Kummerlogform}
		Write
		$D=\sum_{\lambda \in \Lambda}n_{\lambda}D_{\lambda}$,
		where $D_{\lambda}$ ($\lambda\in \Lambda$) are irreducible components of $D$.
		Then we have an exact sequence
		\[
		0\to \WXDplog {m-1} r \xrightarrow{\ul{p}} \WXDlog m r \to\WXDlog 1 r \to 0,
		\]
		where $[D/p]=\sum_{\lambda \in \Lambda}[n_{\lambda}/p]D_{\lambda}$ with
		$[n/p]=\min\{n'\in \Z|\;pn'\geq n\}.$
	\end{theorem}
	
	\begin{proof}
		The claim follows from Theorem \ref{Th.nuKummer} below by the isomorphism (\ref{relativedlog}).
	\end{proof}
	
	Let $R$ be the henselization of a local ring of a smooth scheme over a field $k$ of characteristic $p>0$.
	Let $(T_1,\dots,T_d)\subset R$ be a part of a system of regular parameters and put $T=T_1\cdots T_d$.
	We endow $\bN^d$ with a semi-order by
	\[
	(n_1,\dots,n_d)\leq (n'_1,\dots,n'_d)\; \text{if}\; n_i\leq n'_i \;\;\text{for} \;\; \forall i
	\]
	and put
	\[
	\ul{1}=(1,\dots,1).
	\]
	Following \cite[\S4]{blochkato},
	we define $\FKMRr \eb\subset K^M_r(R)$ for $\eb=(n_1,\dots,n_d)\in \bN^d$ as the subgroup generated by symbols
	\[
	\{x_1,\dots,x_r\}\qwith x_1\in 1+T_1^{n_1}\cdots T_d^{n_d}R, \; x_i\in R[1/T]^\times\; (2\leq i\leq d).
	\]
	(Here having the injectivity of $K^M_r(R)\to K^M_r(K)$ with the quotient field $K$ of $R$, the above symbols are considered in $K^M_r(K)$.)
	For an integer $m>0$, put
	\[
	\FkMRr \eb m =\Image(\FKMRr \eb \to \KMRr/p^m).
	\]
	
	\begin{theorem}\label{Th.nuKummer}
		We have the following exact sequence:
		\[
		0\to \FkMRr \ebp {m-1} \xrightarrow{p} \FkMRr \eb m \to \FkMRr \eb 1 \to 0,
		\]
		where $\ebp=\min\{\ul{\nu}\in \bN^d|\; p\ul{\nu}\geq \eb\}\in \bN^d$.
	\end{theorem}

	For the proof we compute
	\[
	\grikMRr \eb m = \FkMRr \eb m/\FkMRr {\eb+\delta_i} m\qwith
	\delta_i=(0,\dots,\overset{\underset{\vee}{i}}{1},\dots,0).
	\]
	We need some preliminaries. For $\eb\in \bN^d$ and $1\leq i\leq d$ and an integer $q\geq 1$ put
	\[
	\wni q  = I^\eb\Omega^q_R(\log T)\otimes_R R_i \qwith R_i=R/(T_i),
	\]
	where $I^\eb=(T_1^{n_1}\cdots T_d^{n_d})\subset R$ and
	$\Omega^q_R(\log T)$ is the sheaf of (absolute) differential $q$-forms of $R$ with logarithmic poles along $T=0$.
	It is easy to check the exterior derivative induces
	\[
	d^q:\wni q \to \wni {q+1} .
	\]
	Put
	\[
	\Zni q =\Ker(\wni q \rmapo{d^q} \wni {q+1}),\quad
	\Bni q =\Image(\wni  {q-1} \rmapo{d^{q-1}} \wni {q}).
	\]
	We can easily check the following.
	
	\begin{lemma}(\cite[Thm. 2.16]{ruellingsaito})\label{lem-wniCartier}
		Let the notation be as above. Then the inverse Cartier morphism
		\[
		C^{-1}: \Omega^q_R \to \Omega^q_R/d\Omega^{q-1}_R
		\]
		induces an isomorphism
		\[
		\Cni:\wnpi q \isom  \Zni q/\Bni q.
		\]
	\end{lemma}
	
	We define subgroups
	\[
	\Bni q =\BBni q 1 \subset \BBni q 2\subset \cdots\subset
	\ZZni q 2\subset \ZZni q 1=\Zni q\subset \wni q,
	\]
	by the inductive formula
	\[
	\BBnpi q s \rmapou{\Cni}{\simeq} \BBni q {s+1}/\Bni q,\quad
	\ZZnpi q s \rmapou{\Cni}{\simeq} \ZZni q {s+1}/\Bni q.
	\]
	\medbreak
	
	\begin{proposition}\label{grKM}
		Fix $\eb=(n_1,\dots,n_d)\in \bN^d$ and $1\leq i\leq d$.
		\begin{itemize}
			\item[(1)]
			There is a natural map
			\[
			\rho_{\eb,i}: \wni {r-1} \to \grikMRr \eb m
			\]
			such that for $a\in R$, $b_2,\dots,b_d \in R[1/T]^\times$,
			\[
			\rho_{\eb,i}\big(a(T_1^{n_1}\cdots T_d^{n_d}) \dlog{b_2}\wedge\cdots\wedge \dlog{b_r}\big)=
			\{1+aT_1^{n_1}\cdots T_d^{n_d},b_2,\dots,b_r\}\;\in \; \FKMRr \eb.
			\]
			\item[(2)]
			Write $n_i=p^s\cdot n'$ with $p\not|n'$. If $m>s$, $\rho_{\eb,i}$ induces an isomorphism
			\[
			\wni {r-1}/\BBni {r-1} s \isom \grikMRr \eb m
			\]
			If $m\leq s$, $\rho_{\eb,i}$ induces an isomorphism
			\[
			\wni {r-1}/\ZZni {r-1} m \isom \grikMRr \eb m
			\]
		\end{itemize}
	\end{proposition}
	\begin{proof}
		The existence of $\rho_{\eb,i}$ together with the fact that it induces the surjective maps as in (2) is shown
		by the same argument as \cite[(4.5) and (4.6)]{blochkato}. Note that $\wni {r-1}/\BBni {r-1} s$ and $\wni {r-1}/\ZZni {r-1} m$ are free $R_i^{p^e}$-modules, for some $e>>0$.  By localization, the injectivity of the maps is reduced
		to the case $R$ is a discrete valuation ring, which has been treated in \cite[(4.8)]{blochkato}.
	\end{proof}
	\bigskip

	Now we prove Theorem \ref{Th.nuKummer}.
	It is easy to see that we have a complex as in the theorem. 
	Its exactness on the left follows from the fact that $K^M_r(R)$ is $p$-torsion free (cf. \cite[Thm. 8.1]{geisserlevine} and \cite[Thm. 6.1]{rost1996chow}).
	It remains to show the exactness in the middle. For this it suffices to show the injectivity of the map induced by multiplication by $p$:
	\[
	K^M_r(R)/\FKMRr {\ebp} +p^{m-1}K^M_r(R) \rmapo {p} K^M_r(R)/\FKMRr {\eb} +p^{m}K^M_r(R).
	\]
	
	This follows from the following claims.
	
	\begin{claim}\label{claim1}
		The multiplication by $p$ induces an injective map:
		\[
		K^M_r(R)/\FKMRr {\ul{1}}+p^{m-1}K^M_r(R)  \to K^M_r(R)/\FKMRr {\ul{1}} +p^{m}K^M_r(R).
		\]
	\end{claim}
	\begin{proof}
		We have a map (cf. \cite[the first displayed formular in the proof of Prop. 2.10]{ruellingsaito})
		\[K^M_r(R) /\FKMRr {\ul{1}}\to \underset{1\leq i\leq d}{\bigoplus}\; K^M_r(R_i); \{a_1,\cdots, a_r\}\mapsto \oplus_i\{{a}_1 \ \mathrm{mod}\ T_i,\cdots,a_r \ \mathrm{mod}\ T_i\}, \]
		where $(a\ \mathrm{mod}\ T_i) \in R_i$ is the image of $a\in R$. By Prop. 2.10 in loc.cit. and Proposition \ref{milnor.transition}, we see that this map is injective. Combining with the fact that  $\bigoplus_{1\leq i\leq d}K^M_r(R_i)$  is $p$-torsion free, we conclude this claim.
	\end{proof}
	
	\begin{claim}\label{claim2}
		For $\eb$ and $i$ as in Proposition \ref{grKM}, the multiplication by $p$ induces an injective map:
		\[
		\grikMRr \ebp {m-1}  \to \grikMRr \eb {m}  .
		\]
	\end{claim}
	\begin{proof}
		It is easy to check that the multiplication by $p$ induces such a map.
		Its injectivity follows from the commutative diagram
		\[
		\begin{CD}
		\wnpi {r-1}/\BBnpi {r-1} {s-1} @>{\Cni}>> \wni {r-1}/\BBni {r-1} {s}\\
		@VV{\simeq}V @VV{\simeq}V\\
		\grikMRr \ebp {m-1} @>>> \grikMRr \eb {m} \\
		\end{CD}
		\quad\text{ if } m>s,
		\]
		and the commutative diagram
		\[
		\begin{CD}
		\wnpi {r-1}/\ZZnpi {r-1} {m-1} @>{\Cni}>> \wni {r-1}/\ZZni {r-1} {m}\\
		@VV{\simeq}V @VV{\simeq}V\\
		\grikMRr \ebp {m-1} @>>> \grikMRr \eb {m} \\
		\end{CD}
		\quad\text{ if } m\leq s,
		\]
		where the vertical isomorphisms are from Proposition \ref{grKM}.
	\end{proof}

	\subsection{Relation with differential forms}
	
	The sheaf $\Omega_{X|D,\log}^r$  relates to coherent sheaves as follows.
	\begin{theorem} \label{logtocoherent}
		We have an exact sequence
		\[ 0 \to \Omega_{X|D,\log}^r \to \Omega_{X|D}^r \xrightarrow{1-C^{-1}} \Omega_{X|D}^r/d\Omega_{X|D}^{r-1} \to 0,\]
		where $\Omega_{X|D}^r=\Omega_{X/k}^r(\log D)\otimes_{\mathcal{O}_X}\mathcal{O}_X(-D)$.
	\end{theorem} 
	\begin{proof}
		
		For the exactness on the right, it suffices to show the surjectivity of $1-C^{-1}$ on sections over the strict henselization of a local ring of $X$. In fact, by the argument in the classical case where $D=\emptyset$ (\cite[Lem. 1.3]{milnesurface}), it suffices to show the following claim.
		\begin{claim}\label{multiplicity}
			Let $A$ be a strictly henselian regular local ring of equi-characteristic $p>0$
			and $\fm\subset A$ be the maximal ideal. 
			Let $\pi\in \fm$ and $a\in A$. If $a\in \pi A$, then there exists $b\in A$, such that $b\in \pi A$ and $b^p-b=a$.   
		\end{claim}
		\begin{proof}[Proof of Claim \ref{multiplicity}]\renewcommand{\qedsymbol}{}
			Let $k$ be the residue field of $A$. 
			Since $\phi: A\to A$ is surjective, there exists $\tilde{b}\in A$ such that $\tilde{b}^p-\tilde{b}=a$. Letting $\beta\in k$ be the image of $\tilde{b}$, 
			$\beta^p-\beta=0\in k$ by the assumption $a\in \pi A\subset\fm$.
			Hence $\beta\in \Fp\subset A$ and we put $b=\tilde{b}-\beta\in A$.
			Then 
			\[b(b^{p-1}-1)=b^p-b=\tilde{b}^p-\tilde{b}=a\in \pi A.\]
			Since $b\in \fm_A$ by the construction, $b^{p-1}-1\in A^\times$
			and we get $b\in \pi A$.
		\end{proof}
		It remains to show the exactness in the middle, i.e., to show that  $\Omega_{X|D}^r\cap \Omega_{X,\log}^r=\Omega_{X|D,\log}^r$. This is a \'etale local question, which is a consequence of Proposition \ref{logform2.prop} below, which is a refinement of \cite[Prop.1]{katogalois}.
	\end{proof}

	Let $R$ be the henselization of a local ring of $X$ and choose a system $T_1,\dots,T_d$ of regular parameters of $R$ such that
	$\Supp(D)=\Spec(R/(T_1\cdots T_e))\subset \Spec(R)$ for some $e\leq d=\text{dim}(R)$. 
	Let $\DR 1 R(\log D)$ denotes the module of differentials with logarithmic poles along $D$ and put $\DR q R(\log D)=\overset{q}{\wedge}\; \DR 1 R(\log D)$.
	For a tuple of integers $\eb=(n_1,\dots,n_e)$ with $n_i \geq 1$, put
	\[
	\FDR {\eb}q R=(T_1^{n_1}\cdots T_e^{n_e})\cdot \DR q R(\log D)\;\subset\; \DR q R,
	\]
	\[
	\FnuR \eb q = \Ker\big(\FDR {\eb}q R \rmapo{1-C^{-1}} \DR q R(\log D)/d\DR {q-1} R(\log D)\big).
	\]
	
	\begin{proposition}\label{logform2.prop}
		$\FnuR \eb q$ is generated by elements of the form
		\[
		\dlog {x_1}\wedge \cdots\wedge \dlog {x_q}\qwith x_1\in 1+(T_1^{n_1}\cdots T_e^{n_e}),
		\; x_i\in R[\frac{1}{T_1\cdots T_e}]^\times \; (2\leq i\leq q).
		\]
	\end{proposition}
	\begin{proof}
		The following argument is a variant of Part (B) of the proof of \cite[Prop.1]{katogalois} (see page 224).
		By \cite{artinapprox}, we may replace $R$ by $R=k[[T_1,\dots,T_d]]$.
		Indeed, to use Artin approximation we have to equip any $R$-algebra with the log structure coming via pullback from the canonical one on $(R, D)$ to extend the group $\FnuR \eb q$ to a functor on $R$-algebras $S\mapsto G^{\un}\nu_S(q)$.
		Put $A=k[[T_1,\dots,T_{d-1}]]$ and $T=T_d$ so that $R=A[[T]]$.
		Let $\DR q A(\log E)$ be the module of differential $q$-forms on $\Spec(A)$ 
		with logarithmic poles along $E=\Spec(A/(T_1\cdots T_{d-1})) \subset \Spec A$.
		By loc. cit., we have an isomorphism
		\begin{equation}\label{logform2.eq0}
		(R \otimes_A \DR q A(\log E)) \;\oplus \; (R \otimes_A \DR {q-1} A(\log E))\; \simeq \; \DR q R(\log D)\;;\;
		(a\otimes w,b\otimes v) \;\to\; aw + bv\wedge \dlog T.
		\end{equation}
		For each $n\geq 1$, let $V_n\subset \DR q R(\log D)$ be the image of
		\[
		(T^n A[[T]]\otimes \DR q A(\log E)) \;\oplus \; (T^n A[[T]] \otimes_A \DR {q-1} A(\log E)).
		\]
		We easily check the following.
		
		\begin{claim}\label{logform2.claim1}
			For a tuple of integers $\eb=(n_1,\dots,n_{d-1},n)$ with $n,n_i \geq 1$, we have
			$\FDR \eb q R\subset V_n$ and it coincides with the image of
			\[
			(T^n A[[T]] \otimes_A (T_1^{n_1}\cdots T_{d-1}^{n_{d-1}})\cdot\DR q A(\log E)) \;\oplus \; 
			(T^n A[[T]]\otimes_A (T_1^{n_1}\cdots T_{d-1}^{n_{d-1}})\cdot\DR {q-1} A(\log E)).
			\]
			The map \eqref{logform2.eq0} restricted on $V_n$ induces an isomorphism
			\[
			(T_1^{n_1}\cdots T_{d-1}^{n_{d-1}})\cdot\DR q A(\log E) \;\oplus \; 
			(T_1^{n_1}\cdots T_{d-1}^{n_{d-1}})\cdot\DR {q-1} A(\log E) \isom
			\FDR \eb q R/\FDR {\eb'} q R,
			\]
			\[
			(w,v) \quad \to \quad T^n(w + v\wedge \dlog T).
			\]
			where $\eb'=(n_1,\dots,n_{d-1},n+1)$.
		\end{claim}
		
		Let $I_q$ be the set of strictly increasing functions $\{1,\dots,q\} \to \{1,\dots,d-1\}$.
		For $s\in I_q$ write
		\[
		\omega_s=\dlog {T_{s(1)}}\wedge \cdots\wedge \dlog {T_{s(q)}}\;\in \DR q A(\log E).
		\]
		Then $\omega_s$ ($s\in I_q$) form a basis of $\DR q A(\log E)$ over $A$.
		Put 
		\[
		U_n=V_n\cap \Ker\big(\DR q R(\log D) \rmapo{1-C^{-1}} \DR q R(\log D)/d\DR {q-1} R(\log D)\big).
		\]
		We have the following description of $U_n/U_{n+1}$ (see Part (B) of the proof of \cite[Prop.1]{katogalois}). 
		
		If $(p,n)=1$, we have an isomorphism
		\[
		\rho_n: \DR {q-1} A(\log E) \isom U_n/U_{n+1},
		\]
		\begin{equation}\label{logform2.eq1}
		\sum\limits_{s\in I_{q-1}} a_s\omega_s \mapsto \sum\limits_{s\in I_{q-1}}\dlog {(1+a_s T^n)}\wedge \omega_s \;\; (a_s\in A).
		\end{equation}
		
		If $p|n$, we have an isomorphism
		\[
		\rho_n: \DR {q-1} A(\log E)/\DR {q-1} A(\log E)_{d=0}\;\oplus\; \DR {q-2} A(\log E)/\DR {q-2} A(\log E)_{d=0} \isom U_n/U_{n+1},
		\]
		\begin{equation}\label{logform2.eq2}
		(\sum\limits_{s\in I_{q-1}}a_s\omega_s, \sum\limits_{t\in I_{q-2}}b_t\omega_t) \mapsto \sum\limits_{s\in I_{q-1}}\dlog {(1+a_s T^n)}\wedge \omega_s+\sum\limits_{t\in I_{q-2}}\dlog {(1+b_t T^n)}\wedge\dlog T \wedge \omega_t,
		\end{equation}
		where $a_s,b_t\in A$.
		
		\begin{claim}\label{logform2.claim2}
			Fix a tuple of integers $\eb=(n_1,\dots,n_{d-1},n)$ with $n_i \geq 1$.
			\begin{itemize}
				\item[(1)]
				Assume $(p,n)=1$ and $\rho_n(\omega)\in \FDR \eb q R \mod U_{n+1}$ for
				\[
				\omega=\underset{s\in I_{q-1}}{\sum}\; a_s\omega_s\in  \DR {q-1} A(\log E).
				\]
				Then we have $a_s\in (T_1^{n_1}\cdots T_{d-1}^{n_{d-1}})$ for all $s\in I_{q-1}$.
				\item[(2)]
				Assume $p|n$ and $\rho_n(\omega)\in \FDR \eb q R \mod U_{n+1}$ for 
				\[
				\omega=(\omega_1,\omega_2)\;\in \;
				\DR {q-1} A(\log E)/\DR {q-1} A(\log E)_{d=0}\;\oplus\; \DR {q-2} A(\log E)/\DR {q-2} A(\log E)_{d=0}.
				\]
				Then one can write
				\[
				\omega_1=\underset{s\in I_{q-1}}{\sum}\; a_s\omega_s\mod \DR {q-1} A(\log E)_{d=0},
				\]
				\[
				\omega_2=\underset{t\in I_{q-2}}{\sum}\; b_t\omega_t \mod \DR {q-2} A(\log E)_{d=0},
				\]
				with $a_s,b_t\in (T_1^{n_1}\cdots T_{d-1}^{n_{d-1}})$ for all $s\in I_{q-1}$ and $t\in I_{q-2}$.
			\end{itemize}
		\end{claim} 
		\begin{proof}[Proof of Claim \ref{logform2.claim2}]\renewcommand{\qedsymbol}{}
			Assume $(p,n)=1$. From \eqref{logform2.eq1} we get
			\[
			\rho_n(\underset{s\in I_{q-1}}{\sum}\; a_s\omega_s)=
			T^n \underset{s\in I_{q-1}}{\sum}\; da_s\wedge\omega_s\;\pm\;
			nT^n  \underset{s\in I_{q-1}}{\sum}\; a_s \omega_s\wedge \dlog{T} \;\mod U_{n+1}.
			\]
			Hence (1) follows from Claim \ref{logform2.claim1} noting $da_s\wedge\omega_s\in \DR q A(\log E)$. 
			Next assume $p|n$. From \eqref{logform2.eq2} we get
			\[
			\rho_n((\underset{s\in I_{q-1}}{\sum}\; a_s\omega_s,\underset{t\in I_{q-2}}{\sum}\; b_t\omega_t))\;=\;
			T^n \underset{s\in I_{q-1}}{\sum}\; da_s \wedge \omega_s \;\pm\;
			T^n \underset{t\in I_{q-2}}{\sum}\; db_t \wedge \omega_t \wedge\dlog T . 
			\]
			By Claim \ref{logform2.claim1}, if the left hand side lies in $\FDR \eb q R \mod U_{n+1}$, we get
			\[
			da_s \wedge \omega_s\in (T_1^{n_1}\cdots T_{d-1}^{n_{d-1}})\cdot \DR q A(\log E),\quad
			db_t \wedge \omega_t\in (T_1^{n_1}\cdots T_{d-1}^{n_{d-1}})\cdot \DR{q-1} A(\log E).
			\]
			Thus the desired assertion follows from the following.
			
			\begin{claim}
				Assume $d\eta\in (T_1^{n_1}\cdots T_{d-1}^{n_{d-1}})\cdot \DR q A(\log E)$ for 
				$\eta=\underset{s\in I_{q-1}}{\sum}\; a_s\omega_s\in \DR {q-1} A(\log E)$.
				Then there exist $\alpha_s\in A$ for $s\in I_{q-1}$ such that $a_s-\alpha_s\in (T_1^{n_1}\cdots T_{d-1}^{n_{d-1}})$
				for all $s$ and that $d\xi=0$ for $\xi=\underset{s\in I_{q-1}}{\sum}\; \alpha_s\omega_s$.
			\end{claim}
			
			Indeed write $a_s=\alpha_s+a'_s$ where $a'_s\in (T_1^{n_1}\cdots T_{d-1}^{n_{d-1}})$ and 
			$\alpha_s$ are expanded as
			\[
			\underset{i_1,\dots,i_{d-1}}{\sum}\; \alpha_{s,i_1,\dots,i_{d-1}} T_1^{i_1}\cdots T_{d-1}^{i_{d-1}}\;\;
			(\alpha_{s,i_1,\dots,i_{d-1}}\in k),
			\]
			where $i_1,\dots,i_{d-1}$ range over non-negative integers such that there exists $1\leq \nu\leq d-1$ with $i_\nu<n_\nu$.
			Then one easily check that $\alpha_s$ satisfy the desired condition.
		\end{proof}

		Now we can finish the proof of Proposition \ref{logform2.prop}. In the following we fix a tuple of integers $\eb=(n_1,\dots,n_{d-1},n_d)$ with $n_i \geq 1$ and take $\omega\in \FDR\eb q R$.
		By Claim \ref{logform2.claim2} there exist a series of elements
		\[
		\begin{aligned}
		& a_{s,n}\in (T_1^{n_1}\cdots T_{d-1}^{n_{d-1}})\;\; (s\in I_{q-1},\; n\geq n_d),\\
		& b_{t,pm}\in (T_1^{n_1}\cdots T_{d-1}^{n_{d-1}})\;\; (t\in I_{q-2},\; m\geq n_d/p),\\
		\end{aligned}
		\]
		such that
		\[
		\begin{aligned}
		\omega =
		& \underset{n\geq n_d}{\sum}\;\underset{s\in I_{q-1}}{\sum}\; \dlog{(1+a_{s,n}T^n)} \wedge \omega_s \;+\; 
		\underset{pm\geq n_d}{\sum}\;\underset{t\in I_{q-2}}{\sum}\; \dlog{(1+b_{t,m}T^{pm})} \wedge \dlog{T}\wedge \omega_t\\
		=&\underset{s\in I_{q-1}}{\sum}\Big(\underset{n\geq n_d}{\sum}\; \dlog{(1+a_{s,n}T^n)}\Big)
		\wedge \omega_s \;+\; 
		\underset{t\in I_{q-2}}{\sum}\Big(\underset{pm\geq n_d}{\sum}\; \dlog{(1+b_{t,m}T^{pm})}\Big)\wedge \dlog{T}\wedge \omega_t
		\end{aligned}
		\]
		The products
		\[
		x=\underset{n\geq n_d}{\prod}\; (1+a_{s,n}T^n),\quad y=\underset{pm\geq n_d}{\prod}\; (1+b_{t,m}T^{pm})
		\]
		converge in $1+(T_1^{n_1}\cdots T_{d}^{n_{d}})\subset R^\times$ and we get
		\[
		\omega = \underset{s\in I_{q-1}}{\sum}\; \dlog{x} \wedge \omega_s \;+\; 
		\underset{t\in I_{q-2}}{\sum}\; \dlog{y} \wedge \dlog{T}\wedge \omega_t.
		\]
		This completes the proof of Proposition \ref{logform2.prop}.
	\end{proof}
	
	\begin{remark}
		In fact, the above proof shows that the exactness in the middle of the complex in Theorem \ref{logtocoherent} already holds in the Nisnevich topology. 
	\end{remark}

	\bigskip
	
	\section{Filtered de Rham-Witt complexes}
	\bigskip
	
	Let $X, D, j\colon: U\hookrightarrow X$ be as before. Let $\{D_{\lambda}\}_{\lambda\in \Lambda}$ be the (smooth) irreducible components of $D$. We endow $\Z^{\Lambda}$ with a semi-order by 
	\begin{equation}\label{semi.order} \underline n:=\nlam \geq \underline {n'}:=({n}_{\lambda}')_{\lambda\in\Lambda} \; \text{if} \; n_{\lambda} \geq {n}_{\lambda}' \;\text{for all} \; \lambda\in\Lambda.\end{equation}
	For $\underline n=\nlam \in \Z^{\Lambda}$ let
	$$
	\D n=\sum\limits_{\lambda\in\Lambda}n_{\lambda}D_{\lambda}
	$$
	be the associated divisor.
	\subsection{Definition and basic properties}
	Let $E$ be a Cartier divisor on $X$. It is given by $\{V_i,f_i\}$, where $\{V_i\}_i$ is an open cover of $X$ and $f_i\in \Gamma(V_i, \mathcal{M}_X^{\times})$ is a section of the sheaf of total fractional ring.
	\begin{definition}\label{def.filter.divisor}
		We define an invertible $W_m\cO_X$-module $W_m\cO_X(E)$ associated to $E$ as:
		\[ W_m\cO_X(E)_{|V_i}:= W_m\cO_{V_i}\cdot [\frac{1}{f_i}]_m \subset W_m\mathcal{M}_{V_i}, \]
		where $[\cdot]_m\colon \cO \to W_m\cO$ the Teichm\"uller lifting.
	\end{definition}
	
	This definition gives us an invertible sheaf $W_m\cO_X(D_{\un})$ for any $D_{\un}$ as above.
	
	\begin{lemma}\label{FVRon.inv}
		We have \begin{itemize}
			\item[(i)] $F(W_{m+1}\cO_X(D_{\un}))\subset W_{m}\cO_X(D_{p\un})$;
			\item[(ii)] $V(W_m\cO_X(D_{p\un}))\subset W_{m+1}\cO_X(D_{\un})$; 
			\item[(iii)] $R(W_{m+1}\cO_X(D_{\un}))\subset W_{m}\cO_X(D_{\un}) $.
		\end{itemize}
	\end{lemma}
	\begin{proof}
		The claim (i) and (iii) are clear by the definition. For (ii), it follows from the equalities $V(x\cdot Fy)=V(x)\cdot y$ and $F[y]_{m+1}=[y^p]_m$.
	\end{proof}
	Let $W_m\Omega_X^{\ast}(\log D)$ be the de Rham-Witt complex with respect to the canonical log structure $(X, j_*\cO_U^{\times}\cap \cO_X)$ defined in \cite[\S4]{hyodokato}. 
	\begin{definition} For  $\underline n=\nlam \in \Z^{\Lambda}$, we define the filtered de Rham-Witt complex as
		\[W_m\Omega^*_{X|D_{\un}}:=  W_m\cO_X(-D_{\un})\cdot W_m\Omega_X^{\ast}(\log D) \subset j_*W_m\Omega_U^{\ast},\]
		where $W_m\Omega_X^{\ast}(\log D)$ is canonically viewed as a subsheaf of $j_*W_m\Omega_U^{\ast}$(cf. \cite[(4.20)]{hyodokato}). 
	\end{definition}
	Note that \[W_m\Omega^*_{X|D_{\un}} \cong W_m\Omega_X^{\ast}(\log D)\otimes_{W_m\cO_X} W_m\cO_X(-D_{\un}). \]
	In particular, $W_1\Omega^*_{X|D_{\un}}= \Omega^*_{X}(\log D)\otimes \cO_X(-D_{\un})=\Omega^*_{X|D_{\un}} $(cf. notation in Theorem \ref{logtocoherent}). 
	\begin{lemma}\label{FVR.on.filtration}
		We have the following inclusions \begin{itemize}
			\item[(i)] $F(W_{m+1}\Omega^*_{X|D_{\un}})\subset W_{m}\Omega^*_{X|D_{p\un}}$;
			\item[(ii)] $V(W_m\Omega^*_{X|D_{p\un}})\subset W_{m+1}\Omega^*_{X|D_{\un}}$; 
			\item[(iii)] $R(W_{m+1}\Omega^*_{X|D_{\un}})\subset W_{m}\Omega^*_{X|D_{\un}} $.
		\end{itemize}
	\end{lemma}
	\begin{proof}
		This follows from Lemma \ref{FVRon.inv} and the basic properties of de Rham-Witt complex \cite[\S 4.1]{hyodokato} \cite[Prop. 1.5]{lorenzon}.
	\end{proof}
	
	\subsection{Canonical filtration}
	On $W_m\Omega_X^{\ast}(\log D)$, we can define the canonical filtration as in \cite[I (3.1.1)]{illusiederham}:
	\begin{equation*}
	\Fil^sW_m\Omega_X^{r}(\log D):=\begin{cases} \hfil W_m\Omega_X^{r}(\log D),  &\mathrm{if} \; s\leq 0\; \mathrm{or} \; r\leq 0,\\
	\Ker(R^{m-s}: W_m\Omega_X^{r}(\log D) \to W_s\Omega_X^{r}(\log D)),  &\mathrm{if} \  1\leq s\leq m,\\
	\hfil 0, &\mathrm{if} \ s\geq m.
	\end{cases}
	\end{equation*}
	For $1\leq s \leq m$, we have \cite[Prop. 1.16]{lorenzon}:
	\[ \Fil^s W_m\Omega^r_{X}(\log D)=V^sW_{m-s}\Omega^r_{X}(\log D)+dV^sW_{m-s}\Omega^{r-1}_{X}(\log D). \]
	
	\begin{definition}
		For $1\leq s \leq m$, we define
		\begin{equation*}
		\Fil^sW_m\Omega_{X|D_{\un}}^{r}:=\begin{cases} \hfil W_m\Omega_{X|D_{\un}}^{r},  &\mathrm{if} \; s\leq 0\; \mathrm{or} \; r\leq 0,\\
		\Ker(R^{m-s}: W_m\Omega_{X|D_{\un}}^{r} \to W_s\Omega_{X|D_{\un}}^{r}),  &\mathrm{if} \  1\leq s\leq m,\\
		\hfil 0, &\mathrm{if} \ s\geq m.
		\end{cases}
		\end{equation*}

	\end{definition}  
	
	\begin{theorem}\label{standardfiltration}
		We have 
		\[ \Fil^s W_m\Omega^r_{X|D_{\un}}=V^s W_{m-s}\Omega^r_{X|D_{p^s\un}}+dV^sW_{m-s}\Omega^{r-1}_{X|D_{p^s\un}}. \]
		
	\end{theorem}
	\begin{proof}
		We only need to show the inclusion $"\subseteq"$. 
		By the definition of the canonical filtration and the fact that
		$W_m\cO_X(-D_{\un})$ is an invertible sheaf, we have 
		\[ \Fil^s W_m\Omega^r_{X|D_{\un}}=W_m\cO_X(-D_{\un})\cdot \Fil^s W_m\Omega^r_{X}(\log D), \]
		and it suffices to show that the group on the right hand side is contained in
		\[V^s W_{m-s}\Omega^r_{X|D_{p^s\un}}+dV^sW_{m-s}\Omega^{r-1}_{X|D_{p^s\un}}. \]
		Using the formula $ x\cdot Vy=V(F(x)\cdot y)$ repeatedly, we see that for any $\omega \in W_{m-s}\Omega_X^r(\log D)$, $\omega'\in W_{m-s}\Omega_X^{r-1}(\log D)$ and $x\in W_m\cO_X(-D_{\un})$, 
		\begin{equation}\label{left.in.right} x\cdot (V^s(\omega)+dV^{s}(\omega'))=V^s(F^s(x)\cdot \omega)+dV^{s}(F^s(x)\cdot\omega')\pm dx\cdot V^s(\omega'). \end{equation}
		By our definition, we have $F^s(x) \cdot\omega\in W_{m-s}\Omega_{X|D_{p^s\un}}^r$ and $F^s(x) \cdot\omega'\in W_{m-s}\Omega_{X|D_{p^s\un}}^{r-1}$.
		It suffices to prove that $dx\cdot V^s(\omega')\in V^sW_{m-s}\Omega_{X|D_{p^s\un}}^r$. Since the problem is local on $X$, it is enough to show the following claim:
		\begin{claim}
			For any $t\in \cO_X$, and $z'\in W_{m-s}\Omega_X^r(\log D)$,
			\[ d[t]_mV^{s}(z')=V^s([t]_{m-s}^{p^s-1}d[t]_{m-s}z').  \] 
		\end{claim}
		
		Indeed, we know (cf.\cite[I,Prop. 1.5.2]{illusiederham}), 
		\[ d[t]_mV(z)=V([t]_{m-1}^{p-1}d[t]_{m-1}z)\quad\text{for any $t\in \cO_X$, and $z\in W_{m-1}\Omega_X^r(\log D) $}. \]
		Using this formula and $ x\cdot Vy=V(F(x)\cdot y)$ repeatedly, we get the claim.
		
		
	\end{proof}
	\begin{corollary}
		There are the following inclusions \begin{itemize}
			\item[(i)] $F(\Fil^s W_m\Omega^r_{X|D_{\un}})\subset \Fil^{s-1} W_{m-1}\Omega^r_{X|D_{p\un}}$;
			\item[(ii)] $V(\Fil^s W_m\Omega^r_{X|D_{p\un}})\subset \Fil^{s+1} W_{m+1}\Omega^r_{X|D_{\un}}$; 
			\item[(iii)] $R(\Fil^s W_m\Omega^r_{X|D_{\un}})\subset \Fil^s W_{m-1}\Omega^r_{X|D_{\un}}$.
		\end{itemize}
	\end{corollary}
	\begin{proof}
		This follows from Lemma \ref{FVR.on.filtration}, $FV=p=VF$ and $FdV=d$.
	\end{proof}
	
	For $\un\geq \ul{1}$, i.e. $\un\in \bN^{\Lambda}$, we have
	\[ W_m\Omega_{X|D_{\un}}^r \subset W_m\Omega_X^r\]
	Indeed, for $m=1$ this follows from the fact 
	$W_1\Omega_{X|D_{\un}}^r=\Omega^r_X(\log D)(-D_{\un}) \subset \Omega_X^r$. 
	Then the claim follows by induction on $m$ using Theorem \ref{standardfiltration}.

	\begin{lemma}\label{kernelps} 
		For $\un\in \bN^{\Lambda}$, we have
		\[   \Fil^s W_m\Omega^r_{X} \cap W_m\Omega^r_{X|D_{\un}} =\Fil^s W_m\Omega^r_{X|D_{\un}},   \]
		and
		\[\Ker(p^s: \FDRW{m}{r}{\un} \to \FDRW{m}{r}{\un}  ) =\Fil^{m-s}\FDRW{m}{r}{\un}.\]
		In particular the multiplication by $p^s$ induces an injective homomorphism
		\begin{equation}\label{ulp^s}
		\ul{p^s}: \FDRW{m-s}{r}{\un}  \hookrightarrow \FDRW{m}{r}{\un}.
		\end{equation}
		
	\end{lemma}
	\begin{proof}
		
		The first equality follows from the commutative diagram
		\[ \xymatrix{ 
			0\ar[r]& \Fil^s W_m\Omega^r_{X|D_{\un}} \ar[r] \ar@{^(->}[d]&W_m\Omega^r_{X|D_{\un}} \ar[r]^-{R^{m-s}} \ar@{^(->}[d]& W_s\Omega^r_{X|D_{\un}} \ar[r] \ar@{^(->}[d] & 0\\
			0 \ar[r] &\Fil^s W_m\Omega^r_{X} \ar[r] & W_m\Omega^r_{X} \ar[r]^-{R^{m-s}} &W_s\Omega^r_{X} \ar[r]&0 \;.
		}\]
		
		The second equality follows from the first and the fact (cf. \cite[Prop. 3.4]{illusiederham})
		\[ \Ker(p^s: \DRW{m}{r}{X} \to \DRW{m}{r}{X}  ) =\Fil^{m-s}\DRW{m}{r}{X}.\]
		
	\end{proof}

	Recall (cf. the proof of \cite[I, Prop. 3.11, Page 575]{illusiederham}) 
	\begin{align}
	\Ker(F^{m-1}\colon \DRW{m}{r}{X}\to \Omega_{X}^r)=V\DRW{m-1}{r}{X}, \\
	\Ker(F^{m-1}d\colon \DRW{m}{r}{X}\to \Omega_{X}^{r+1})=F\DRW{m+1}{r}{X}. \label{imageF} 
	\end{align}
	We have the following analogues for the filtered de Rham-Witt sheaves.

	\begin{proposition}\label{FundV}
		For $\un \in \bN^{\Lambda}$, we have 
		\begin{itemize}
			\item[(i)]$ \Ker(F^{m-1}\colon \FDRW{m}{r}{\un}\to \Omega_{X|D_{p^{m-1}\un}}^r)=V\FDRW{m-1}{r}{p\un}, $
			i.e., \[ VW_{m-1}\Omega^r_X\cap \FDRW{m}{r}{\un}=V\FDRW{m-1}{r}{p\un};\]
			\item[(ii)] $\Ker(F^{m-1}d\colon \FDRW{m}{r}{p\un}\to \Omega_{X|D_{p^{m}\un}}^{r+1})=F\FDRW{m+1}{r}{\un}$, i.e.,\[ F\DRW{m+1}{r}{X}\cap \FDRW{m}{r}{p\un}= F\FDRW{m+1}{r}{\un}.\]
		\end{itemize}
	\end{proposition}
	\begin{proof}
		This is proved by the same argument as the proof of \cite[I, Prop. 3.11]{illusiederham}, which recall below. 
		
		(i)  For $m=1$ it is trivial. For $m>1$ we have
		\[ \Ker F^{m-1} \subset \Ker p^{m-1}= \Fil^1\FDRW{m}{r}{\un}=V\FDRW{m-1}{r}{p\un}+dV\FDRW{m-1}{r-1}{p\un}.\] by Theorem \ref{standardfiltration} and Lemma \ref{kernelps}. It suffices to show that, for $1\leq s \leq m$, 
		\begin{equation}\label{kerFind}
		(\Ker F^{m-1}) \cap (V\FDRW{m-1}{r}{p\un}+\Fil^s\FDRW{m}{r}{\un})  \subset  V\FDRW{m-1}{r}{p\un}+\Fil^{s+1}\FDRW{m}{r}{\un}.
		\end{equation} 
		Let 
		$z=Vx+dV^sy$ with $x\in \FDRW{m-1}{r}{p\un}$, $y \in \FDRW{m-s}{r-1}{p^s\un}$ be such that $F^{m-1}z=0$. 
		Noting $F^{m-1}Vx=pF^{m-2}x=0$ and $F^{m-1}dV^s=F^{m-1-s}d$,
		it follows that $F^{m-1-s}dy=0$. Let $\overline{y}$ be the image of $y$ in $\Omega_{X|D_{p^s\un}}^{r-1}$ under the restriction map $R^{m-1-s}$. 
		Then, by \cite[I, Prop. 3.3]{illusiederham}, we get
		$C^{-(m-1-s)}d\overline{y}=0$ and $d\overline{y}=0$ in $\Omega_{X|D_{p^s\un}}^{r}$.
		By Lemma \ref{lem-wniCartier}
		there exists (locally) $y'\in \Omega_{X|D_{p^{s-1}\un}}^{r-1} $ such that 
		$\overline{y}=C^{-1}(y')$. We can then take a lift $\tilde{y}$ of $y'$ in 
		$\FDRW{m+1-s}{r-1}{p^{s-1}\un}$. 
		Indeed, writing $y'=\sum_{\alpha} a_{\alpha}\omega_\alpha$ with
		$a_{\alpha}\in \cO_X(-D_{p^{s-1}\un})$ and $\omega_\alpha\in \Omega^{r-1}_X(\log D)$,
		we take
		$\tilde{y}=\sum_{\alpha}[a_{\alpha}]_{m+1-s}\tilde{\omega}_\alpha$, where
		$\tilde{\omega}_\alpha\in W_{m+1-s}\Omega^{r-1}_X(\log D)$ is a lift of $\omega_\alpha$. By the construction we have
		\[ y=F\tilde{y} \;\;\mathrm{mod} \; \Fil^1\FDRW{m-s}{r-1}{p^s\un}.\]
		By taking $V^s$ on both sides, we get
		\[ V^sy=V^sF\tilde{y} \;\;\mathrm{mod} \; \Fil^{s+1}\FDRW{m}{r-1}{\un}.\]
		Hence \[ dV^sy=dV^sF\tilde{y}=pdV^{s-1}\tilde{y}=VdV^{s-2}\tilde{y} \; \; \mathrm{mod}\; \Fil^{s+1}\FDRW{m}{r-1}{\un}.\]
		That is \[ dV^sy\in V\FDRW{m-1}{r}{p\un}+\Fil^{s+1}\FDRW{m}{r}{\un}.\]
		Hence $z=Vx+dV^sy\in V\FDRW{m-1}{r}{p\un}+\Fil^{s+1}\FDRW{m}{r}{\un}$, which proves \eqref{kerFind}.
		
		(ii) It suffices to prove that, for $1\leq s\leq m$,
		\begin{equation}\label{kerFdind}
		\Ker (F^{m-1}d) \cap \Fil^s\FDRW{m}{r}{p\un}  \subset  F\FDRW{m+1}{r}{\un}+\Fil^{s+1}\FDRW{m}{r}{p\un}.
		\end{equation} 
		Let $z=V^sx+dV^sy$ with $x \in \FDRW{m-s}{r}{p^{s+1}\un}$, 
		$y \in \FDRW{m-s}{r-1}{p^{s+1}\un} $ be such that $F^{m-1}dz=0$.
		Noting $F^{m-1}dV^s=F^{m-1-s}d$, it follows that $F^{m-1-s}dx=0$. Let $\overline{x}$ be the image of $x$ in $\Omega_{X|D_{p^{s+1}\un}}^r$. As in (i), there exist $\tilde{x}\in \FDRW{m-s+1}{r}{p^sn} $, such that 
		\[ x=F\tilde{x} \;\; \mathrm{mod}\; \Fil^1\FDRW{m-s}{r}{p^{s+1}\un}. \]
		By taking $V^s$ on both sides, we obtain
		\[ V^sx = FV^s\tilde{x} \;\; \mathrm{mod}\; \Fil^{s+1}\FDRW{m}{r}{pn}.  \]
		Noting that $dV^sy=FdV^{s+1}y \in F\FDRW{m+1}{r}{\un}$, we obtain the inclusion (\ref{kerFdind}).
	\end{proof}
	
	\begin{corollary}\label{px.implies.x}
		For $\un \in \bN^{\Lambda}$ and $x\in \DRW{m-1}{r}{X}$, 
		$\ul{p}\cdot x \in \FDRW{m}{r}{\un}$ (cf. \eqref{ulp^s}) implies $x\in \FDRW{m-1}{r}{\un}$.
	\end{corollary}
	\begin{proof}
		Recall we have the following diagram \cite[Prop. 3.4]{illusiederham}:
		\[ \xymatrix{
			\DRW{m}{r}{X} \ar[r]^-{p} \ar[d]^R &  \DRW{m}{r}{X}\\
			\DRW{m-1}{r}{X} \ar@{^(->}[ur]_-{\ul{p}}. &
		}\]
		Hence there exists $\tilde{x} \in \DRW{m}{r}{X}$ such that $p\tilde{x}=\ul{p}\cdot x$ and $R\tilde{x}=x$. By the assumption, we have $VF\tilde{x}=p\tilde{x}=\ul{p}\cdot x\in \FDRW{m}{r}{\un}$. Thanks to Corollary \ref{FundV}(i), it follows that there exists $y'\in \FDRW{m-1}{r}{p\un}$ such that \[ VF\tilde{x}=Vy'. \]
		Recall the identity in \cite[I. 3.21.1.4]{illusiederham}:
		\begin{equation*}
		\Ker(V: \DRW{m-1}{r}{X}\to \DRW{m}{r}{X})=FdV^{m-1}\Omega_X^{r-1}.
		\end{equation*}
		Therefore there exists $z'\in \Omega_X^{r-1}$ such that $F\tilde{x}-y'=FdV^{m-1}z'$. That is 
		\begin{equation*}
		F(\tilde{x}-dV^{m-1}z')=y' \in \FDRW{m-1}{r}{p\un}.
		\end{equation*}
		Corollary \ref{FundV}(ii) implies that there exists $y'' \in \FDRW{m}{r}{\un}$ such that 
		\begin{equation*}
		F(\tilde{x}-dV^{m-1}z')=Fy''.
		\end{equation*} 
		Thanks to the identity \cite[I. 3.21.1.2]{illusiederham}:
		\begin{equation*}
		\Ker(F:\DRW{m}{r}{X} \to \DRW{m-1}{r}{X} )=V^{m-1}\Omega_X^r,
		\end{equation*} 
		we find $z''\in \Omega_X^r$ such that 
		\begin{equation*}
		\tilde{x}-y''=dV^{m-1}z'+V^{m-1}z''.
		\end{equation*}
		Noting that $\Ker(R: \DRW{m}{r}{X} \to \DRW{m-1}{r}{X})=V^{m-1}\Omega_X^r+dV^{m-1}\Omega_X^{r-1}$, we get 
		\begin{equation*}
		x=R\tilde{x}=Ry'' \in \FDRW{m-1}{r}{\un}. 
		\end{equation*} 
		
	\end{proof}

	\subsection{Logarithmic part of filtered de Rham-Witt complexes}\label{defn.filterdRW}

	The relation between the filtered de Rham-Witt sheaves and the relative logarithmic de Rham-Witt sheaves is given by the following theorem, which is a generalization of Theorem \ref{logtocoherent}.
	
	We first introduce some notations. Let $$\Sigma:=\{ \D n \;|\; \un=\nlam \in \bN^{\Lambda}  \}$$
	be the set of effective divisors with supports in $X-U$, whose irreducible components are same as $D$'s. The semi-order on $\Z^{\Lambda}$ defined in (\ref{semi.order}) induces a semi-order on $\Sigma$: 
	\begin{equation*}
	\D n \geq \D {n'}    \; \text{if} \; \underline n \geq \underline{n'}. 
	\end{equation*}
	
	For $D_1, D_2 \in \Sigma$ with $D_1\geq D_2$, we have a natural injective map $\DRWlog m r {X|D_1} \hookrightarrow \DRWlog m r {X|D_2} $ (see Corollary \ref{transition.injective}), which gives a pro-system of sheaves
	\[\lq\lq \varprojlim\limits_{D \in \Sigma}" W_m\Omega_{X|D, \log}^{r}.\]
	In order to simplify the notation, we simply write it as $\lq\lq \varprojlim\limits_{D}" W_m\Omega_{X|D, \log}^{r}$.
	
	\begin{theorem}\label{exactproobj}
		We have the following exact sequence of pro-sheaves
		\[  0 \to \lq\lq \varprojlim_D" \DRWlog{m}{r}{X|D} \to \lq\lq \varprojlim_D" \DRW{m}{r}{X|D} \xrightarrow{1-F} \lq\lq \varprojlim_D" \DRW{m}{r}{X|D}/dV^{m-1}\Omega^{r-1}_{X|p^{m-1}D} \to 0,\]
		where $D$ runs over the set $\Sigma$.
	\end{theorem}
	

	
	We need the following lemma, which follows from easy calculations with Witt vectors.
	\begin{lemma}\label{lem;giesserhe} (\cite[Lem. 1.2.3]{geisserhe})
		Let $R$ be any ring, and $t\in R$, then $ [1+t]_{m}-[1]_m=(y_0,\cdots, y_{m-1})$ with $y_i\equiv t$ mod $ t^2R$ for $0 \leq i \leq m-1.$ Here $[x]_m=(x,0,\cdots,0) \in W_m(R)$ is the Teichm\"uller representative of $x\in R$.
	\end{lemma}
	
	\begin{proof}[Proof of Theorem \ref{exactproobj}.]

		First we show that
		$\RDRWlog{m}{r}{p^{m-1}\ul n} \subset j_\ast W_m \Omega^r_{U,\log}$ 
		(cf. Definition \ref{relative.definition}) lies in $\FDRW m r {\ul n}$.
		This is a local question so that we may assume that $X=\Spec(A)$ and $D=(t)$ for some $t\in A$. 
		By Lemma \ref{lem;giesserhe} we can write
		\[ [1+t^{p^{(m-1)}n}a]_{m}-[1]_m=(t^{p^{(m-1)}n}y_0,\cdots, t^{p^{(m-1)}n}y_{m-1}) \]
		with  $y_{i} \in A$ for $0\leq i \leq m-1$. 
		Noting $dx = 0$ for $x\in W_m(\mathbb F_p)$, we get
		\[  d[1+t^{p^{(m-1)}n}a]_{m} =  d(t^{p^{(m-1)}n}y_0,\cdots, t^{p^{(m-1)}n}y_{m-1})=d([t]_m^{n}\cdot (c_0,\cdots, c_{m-1}))\]
		with  $c_{i} \in A$ for $0\leq i \leq m-1$, where the second equality follows from 
		the formula
		\[[t]_m^{n} \cdot (c_0,\ldots,c_{m-1})=(t^n c_0,t^{np} c_1,\ldots, 
		t^{n p^{(m-1)}}c_{m-1}).\]
		Hence we get
		\begin{equation}\label{dlog.Witt} d\log[1+t^{p^{m-1}n}]_m=([1+t^{p^{m-1}n}]_m)^{-1} d([t]_m^n(c_0,\cdots,c_{m-1})) \in \FDRW m 1 {\ul n},\end{equation}
		noting $[1+t^{p^{m-1}n}]_m$ is a unit of $W_m\cO_X$.
		
		The surjectivity of $1-F$ as pro-systems follows by the same argument of
		the proof of \cite[Prop. 3.26]{illusiederham}.
		Indeed, as in loc. cit., the formula $dx=(F-1)(dVx+dV^2x+\cdots+dV^{m-1}x)$ implies that  
		\[ dW_m\Omega_{X|D_{\un}}^{r-1} \subset (1-F)(W_m\Omega_{X|D_{[n/p^m]}}^r).\]
		Therefore it is enough to show that
		\[ \FDRW{m}{r}{n} \xrightarrow{1-F} \FDRW{m}{r}{n}/dW_m\Omega_{X|D_{\un}}^{r-1} \]
		is surjective. 
		
		Theorem \ref{standardfiltration} implies that $\FDRW{m}{r}{n}/dW_m\Omega_{X|D_{\un}}^{r-1}$ is generated by sections 
		\[V^i[x]_{m-i}d\log[y_1]_{m}\cdots d\log[y_r]_m\;\;\text{with }
		0 \leq i\leq m-1,\]
		where
		$x\in \cO_X(-D_{p^i\un'})$ for some $\un'\leq \un$ and
		$y_j\in \cO_X^{\times}$ for $1\leq j \leq r$ such that
		\[d\log[y_1]_{m}\cdots d\log[y_r]_m\in \DRW{m}{r}{X|D_{\un-\un'}}.\] 
		(Note that in view of \eqref{dlog.Witt}, $d\log[y_i]_m$ may also contribute to the multiplicity.)
		We may then choose (\'etale locally) $y\in \cO_X(-D_{p^i\un'}) $ such that $y^p-y=x$.
		Then we have 
		\[
		\begin{aligned}
		(1-F)\big(V^i[y]_{m-i}d\log[y_1]_{m}\cdots d\log[y_r]_m\big)
		= V^i[x]_{m-i}d\log[y_1]_{m}\cdots d\log[y_r]_m.
		\end{aligned}\]
		which implies the desired surjectivity.
		
		Finally we show the exactness in the middle. 
		It suffices to show the following equality in $\DRW{m}{r}{X}$:
		\[ \FDRW{m}{r}{\ul n} \cap \DRWlog m r X  = \RDRWlog m r {p^{m-1}\ul n}.\]
		We prove this by induction on $m$. For $m=1$, this is Theorem 1.2.1.

		Let $x\in \FDRW{m}{r}{\ul n} \cap \DRWlog m r X  $, then we have
		\[ Rx=Fx \in \FDRW{m-1}{r}{p\un} \cap \DRWlog {m-1} r X.\]	
		By induction hypothesis, we have 
		\[ Rx=Fx \in \RDRWlog{m-1}{r}{p^{m-1}\un} .\]
		
		On the other hand, there is a commutative diagram 
		\[ \xymatrix{ 
			&& \RDRWlog{m}{r}{p^{m-1}\un} \ar[r]^R \ar@{^(->}[d] & \RDRWlog{m-1}{r}{p^{m-1}\un} \ar[r] \ar@{^(->}[d] &0\\
			0\ar[r] & \Omega_{X,\log}^r \ar[r]^{\ul{p}^{m-1}} & \DRWlog{m}{r}{X} \ar[r]^R & \DRWlog{m-1}{r}{X} \ar[r] &0
		}\]
		the lower sequence is exact by \cite[Lem. 3]{colliottorsion}.
		Hence there exist $y\in \RDRWlog{m}{r}{p^{m-1}\un} $ and $z\in \Omega_{X,\log}^r$, such that $x-y= \ul{p}^{m-1}\cdot z$.
		
		Since $\ul{p}^{m-1}\cdot z =x-y \in \FDRW{m}{r}{\ul n}$,
		Corollary \ref{px.implies.x} implies $z\in \Omega_{X|D_{\un}}^r$.
		By Theorem 1.2.1, this implies $z\in \Omega^r_{X|D_{\un},\log}$ and hence
		$\ul{p}^{m-1}\cdot z \in \RDRWlog{m}{r}{p^{m-1}\un}$ (cf. Theorem \ref{Kummerlogform}). This proves $x=y+ \ul{p}^{m-1}\cdot z\in \RDRWlog{m}{r}{p^{m-1}\un}$ as desired.
		
	\end{proof}

	\bigskip
	\section{The pairing on the relative logarithmic de Rham-Witt sheaves}
	\bigskip
	Let $X, D, j\colon U\hookrightarrow X$ be as in \S2.  In the following we want to define a pairing between cohomology group of $\DRWlog{m}{r}{U}$ and cohomology group of $\lq\lq \varprojlim\limits_D" W_m\Omega_{X|D, \log}^{d-r}$.  In order to define a pairing on the sheaves level, we have to write $\DRWlog{n}{r}{U}$ as ind-system of sheaves on $X$.
	
	\subsection{The pairing}
	To define our desired pairing, we will use the notation of two-term complexes. Let's briefly recall the definition. In   \cite{milneduality},  Milne defined a pairing of two-term complexes as follows:
	
	Let $$\mathscr{F}^{\bullet}=(\mathscr{F}^0 \xrightarrow{d_{\mathscr{F}}} \mathscr{F}^1),  \quad \mathscr{G}^{\bullet}=(\mathscr{G}^0 \xrightarrow{d_{\mathscr{G}}} \mathscr{G}^1)$$
	and $$ \mathscr{H}^{\bullet}=(\mathscr{H}^0 \xrightarrow{d_{\mathscr{H}}} \mathscr{H}^1) $$
	be two-term complexes.  A pairing of two-term complexes
	\[\mathscr{F}^{\bullet} \times \mathscr{G}^{\bullet}  \to \mathscr{H}^{\bullet}\]
	is a system of pairings  \[  \langle\; ,\, \rangle_{0,0}^0 :  \mathscr{F}^0 \times \mathscr{G}^0 \to \mathscr{H}^0;    \]
	\[  \langle\; ,\, \rangle_{0,1}^1 :  \mathscr{F}^0 \times \mathscr{G}^1 \to \mathscr{H}^1;    \]
	\[  \langle\; ,\, \rangle_{1,0}^1 :  \mathscr{F}^1 \times \mathscr{G}^0 \to \mathscr{H}^1,   \]
	such that 
	\begin{equation}\label{compatibility}  d_{\mathscr{H}} (\langle x, y\rangle_{0,0}^0)=\langle x, d_{\mathscr{G}}(y)\rangle_{0,1}^1+\langle d_{\mathscr{F}}(x), y\rangle_{1,0}^1 
	\end{equation}
	for all $x\in \mathscr{F}^0$, $y \in \mathscr{G}^0$. Such a pairing is the same as a mapping
	\[  \mathscr{F}^{\bullet} \otimes \mathscr{G}^{\bullet} \to \mathscr{H}^{\bullet}. \]

	In our situation, for any tuple of integers $\underline n \geq \underline 1$ we set 
	
	\begin{equation}\label{finite.WFn} 
	W_m\mathscr{F}^{r,\bullet}_{-\un}:= [ Z_1\FDRW m r {-\ul n} \xrightarrow{1-C} \FDRW m r {-\ul n}],
	\end{equation}
	where $Z_1\FDRW m r {-\un}:=j_*Z_1\DRW{m}{r}{U} \cap \FDRW m r {-\un}$ with $j: U \to X$ the canonical map and
	\[Z_1\DRW{m}{r}{U}:=\Image(F: \DRW {m+1} r U \to \DRW m r U) \stackrel{(\ref{imageF})}{=}\Ker(F^{m-1}d: \DRW m r U \to \Omega_U^{r+1}),\]
	and $C$ is the higher Cartier map \cite[\S 4]{katoI}:\[C\colon Z_1W_{m}\Omega_U^r/dV^{m-1}\Omega_U^{r-1} \xrightarrow{\cong} W_m\Omega_U^r. \]
	We also set 
	\begin{equation}\label{finite.WGn} 
	W_m\mathscr{G}^{d-r,\bullet}_{\un+\ul 1}:= [ \FDRW{m}{d-r}{\ul n+\ul 1}\xrightarrow{1-F} \FDRW m {d-r}{\ul n+\ul 1}/dV^{m-1}\Omega_{X|D_{p^{m-1}\ul n}}^{d-r-1}],   
	\end{equation}
	\begin{equation}\label{finite.WHn} 
	W_m\mathscr{H}^{\bullet}:=[\DRW m d X \xrightarrow{1-C} \DRW m d X]. 
	\end{equation}
	By \cite[Lem. 1.1]{milneduality} we have a canonical isomorphism
	\begin{equation}\label{finite.WHn-log} 
	\DRWlog m d X[0]\simeq W_m\mathscr{H}^{\bullet}.
	\end{equation}
	
	\begin{lemma}
		For any tuple of integers $\underline n \geq \underline 1$ we have a natural pairing of two-term complexes
		\begin{equation}\label{Wm-2-pairing}
		W_m\mathscr{F}^{r,\bullet}_{-\un} \times W_m\mathscr{G}^{d-r,\bullet}_{\un+\ul 1} \to W_m\mathscr{H}^{\bullet}.
		\end{equation}
	\end{lemma}
	\begin{proof}
		By the definition of filtered de Rham-Witt complexes, the cup product induces pairings 
		\[ \FDRW m r {-\ul n}  \times \FDRW m {d-r} {\ul n+\ul 1} \to \FDRW m d {\ul 1} \subset \DRW m d X\]
		and 
		\[ Z_1\FDRW m r {-\ul n}  \times \FDRW m {d-r} {\ul n+\ul 1} \to \FDRW m d {\ul 1} \subset \DRW m d X.\]
		By composing with the higher Cartier operators, we have the following pairing 
		
		\[  Z_1\FDRW m r {-\ul n}  \times \FDRW m {d-r} {\ul n+\ul 1}/dV^{m-1}\Omega_{X|D_{p^{m-1}\ul n}}^{d-r-1}\to \DRW m d X\;;\; (\alpha, \beta) \mapsto -C(\alpha\wedge\beta) .  \] 
		It is easy to see those pairings are compatible.
	\end{proof}
	
	Now let $\underline n$ runs over $\bN^{\Lambda}$, we get a pairing between an ind-object and a pro-object
	\begin{equation}\label{indpropairing}
	\lq\lq \varinjlim_n" W_m\mathscr{F}^{r,\bullet}_{-\un}  \times  \lq\lq \varprojlim_n" W_m\mathscr{G}^{d-r,\bullet}_{\un+\ul 1} \to W_m\mathscr{H}^{\bullet}. 
	\end{equation}
	or equivalently, a morphism in the category of pro-objects of complexes of abelian sheaves
	\[ \lq\lq \varprojlim_n" W_m\mathscr{G}^{d-r,\bullet}_{\un+\ul 1} \to \lq\lq \varprojlim_n" \mathcal{H}\kern -.5pt om(W_m\mathscr{F}^{r,\bullet}_{-\un},W_m\mathscr{H}^{\bullet}),  \]
	where $W_m\mathscr{H}^{\bullet}$ is viewed as a constant pro-object.
	\begin{remark}
		To construct the pairing (\ref{indpropairing}) in a more natural way,
		we can use a full subcategory of the ind-category of pro-objects of coherent complexes(cf. \cite[\S2.1]{kato00}).
	\end{remark}

	\bigskip
	\section{Duality over finite fields}
	\bigskip
	\addtocounter{subsection}{1}
	In this section we assume that  the base field $k$ is finite.  
	By taking hypercohomology groups of the pairing (\ref{indpropairing}) using
	\eqref{finite.WHn-log} , we get a pairing
	
	\begin{equation*}
	\varinjlim_n \mathbb{H}^i(X, W_m\mathscr{F}^{r,\bullet}_{-\un}) \times \varprojlim_n \mathbb{H}^{d+1-i}(X, W_m\mathscr{G}^{d-r,\bullet}_{\un+\ul 1}) \to \mathbb{H}^{d+1}(X, \DRWlog m d X).
	\end{equation*}
	Note that there is an isomorphism in the bounded derived category $D^b(X,\Z/p^m\Z)$ of \'etale $\Z/p^m\Z$-modules:
	\[\varinjlim_nW_m\mathscr{F}^{r,\bullet}_{-\un}=[j_*Z_1\DRW m r U \xrightarrow{1-C} j_*\DRW m r U] \cong  Rj_*\DRWlog m r U,\]
	where the second isomorphism comes from the fact that $j$ is affine.
	Hence we get
	\[ \varinjlim_n \mathbb{H}^i(X, W_m\mathscr{F}^{r,\bullet}_{-\un}) \cong H^i(U, \DRWlog m r U )\; \;\text{for any} \; i\in \Z.\]
	Theorem \ref{exactproobj} implies that 
	\begin{equation}\label{need.p.power} 
	\varprojlim_n \mathbb{H}^{d+1-i}(X, W_m\mathscr{G}^{d-r,\bullet}_{\un+\ul 1}) 
	\cong \varprojlim_D H^{d+1-i}(X, \DRWlog{m}{d-r}{X|D}).
	\end{equation}
	
	Combining these facts, we obtain the following corollary.

	\begin{corollary}
		We have a natural pairing of abelian groups
		\begin{equation}\label{pairing.coh}
		H^i(U, \DRWlog m r U) \times \varprojlim_D H^{d+1-i}(X, \DRWlog m {d-r} {X|D}) \to H^{d+1}(X, \DRWlog m d X) \xrightarrow{\text{Tr}} \Z/p^m\Z,
		\end{equation}
		where the trace map is the canonical trace map of logarithmic de Rham-Witt sheaves (cf.\cite[Cor. 1.12]{milneduality}).
	\end{corollary}

	Noting that $H^{d+1-i}(X, \DRWlog m {d-r} {X|D})$ are finite, we can endow $ \varprojlim\limits_D H^{d+1-i}(X, \DRWlog m {d-r} {X|D})$ with the inverse limit topology, i.e, the profinite topology.
	
	\begin{proposition}
		The pairing is continuous  if we endow $H^i(U, \DRWlog m r U) $ with the discrete topology and $\varprojlim\limits_D H^{d+1-i}(X, \DRWlog m {d-r} {X|D})$ with the profinite topology.
	\end{proposition}
	\begin{proof}
		It suffices to show that the annihilator of each $\alpha \in H^i(U, \DRWlog m r U)$ is open in the projective limit $\varprojlim\limits_D H^{d+1-i}(X, \DRWlog m {d-r} {X|D})$. This follows directly from the lemma below. 
	\end{proof}
	\begin{lemma}
		For any $\alpha \in H^i(U, \DRWlog m r U) $, the morphism induced by \eqref{pairing.coh} 
		\[ \langle \alpha, \cdot \rangle:  \varprojlim\limits_{D} H^{d+1-i}(X, \DRWlog m {d-r} {X|D}) \to H^{d+1}(X, \DRWlog m d X) \]
		factors through $H^{d+1-i}(X, \DRWlog m {d-r} {X|D})$ for some $D\in \Sigma$.
	\end{lemma} 
	\begin{proof}
		This follows directly by the construction of the pairing. 
	\end{proof}

	Our main result in this section is the following duality theorem. 
	\begin{theorem}\label{duality.thm.finite}
		The pairing (\ref{pairing.coh}) is a perfect pairing of topological $\Z/p^m\Z$-modules, i.e, it induces an isomorphism of profinite groups
		\[ \varprojlim_D H^{d+1-i}(X, \DRWlog m {d-r} {X|D}) \xrightarrow{\cong} H^i(U, \DRWlog m r U)^{\vee}, \]
		where $A^\vee$ is the Pontryagin dual of a discrete group $A$. 
	\end{theorem}
	
	The proof is divided into two steps, the first step is to reduce the theorem to the case where $m=1$;  then we prove this special case in the second step.
	
	\begin{proof}
		\textbf{Step 1:} We have the following commutative diagram with exact rows
		\[ \xymatrix@C=.5cm@R=.9cm{  \scriptstyle\cdots \ar[r] & \scriptstyle \varprojlim\limits_D H^{d+1-i}(X, \DRWlog {m-1} {d-r} {X|D}) \ar[d]  \ar[r] &\scriptstyle\varprojlim\limits_D H^{d+1-i}(X, \DRWlog m {d-r} {X|D}) \ar[d]\ar[r] &\scriptstyle \varprojlim\limits_D H^{d+1-i}(X, \DRlog  {d-r} {X|D}) \ar[d] \ar[r] &\scriptstyle \cdots\\
			\scriptstyle \cdots \ar[r] & \scriptstyle H^i(U, \DRWlog {m-1} r U)^{\vee} \ar[r]^{R^{\vee}} & \scriptstyle H^i(U, \DRWlog m r U)^{\vee} \ar[r]^{(\underline{p}^{m-1})^{\vee}} &\scriptstyle H^i(U, \DRlog  r U)^{\vee} \ar[r] &\scriptstyle \cdots,
		}\]
		where the first row is induced by Theorem \ref{Kummerlogform}, and it is exact since the inverse limit is exact for projective system of finite groups. 
		The exactness of the second row is clear.
		Using this commutative diagram and induction on $m$, we reduce our question to the case $m=1$.
		\medbreak
		
		\textbf{Step 2:}
		
		For $m=1$ the pairing (\ref{Wm-2-pairing}) is identified with
		\begin{equation*}
		[Z\DR r {X|D_{-\underline n}}\xrightarrow{\scriptscriptstyle{1-C}} \DR r {X|D_{-\underline n}} ]  \times  [\DR {d-r}{X|D_{\underline n +\underline 1}}\xrightarrow{\scriptscriptstyle{F-1}} \DR {d-r}{X|D_{\underline n+\underline 1}}/d\DR {d-r-1} {X|D_{\underline n+\underline 1}}]  \to [\Omega^d_X \xrightarrow{1-C} \Omega^d_X], 
		\end{equation*}
		where for any $\underline{n}=\nlam \in \mathbb{N}^{\Lambda}$ (cf. the notation in (\ref{finite.WFn})),
		\[\begin{aligned}
		&\DR{r}{X|D_{\underline n}}=\Omega_X^r(\log D)\otimes \cO_X(-D_{\underline n}),\\
		&Z\DR r {X|D_{-\underline{n}}}=
		\Ker(d: \DR{r}{X|D_{\underline n}}\to j_*\DR{r+1}{U}),\;\;
		d\DR {d-r-1} {X|\D n}=\Image(d:\DR {d-r-1} {X|\D n} \to \Omega_X^{d-r} ).\\
		\end{aligned}\]
		
		
		The perfectness of the pairings 
		\begin{equation}\label{perfect.logform.pairing}
		\DR r X (\log D) \otimes \DR {d-r} X (\log D)(-D) \to \DR {d} X (\log D)(-D)=\Omega^d_X. \end{equation}
		implies that the following pairings
		\begin{align}
		\DR r {X|D_{-\underline{n}}} &\otimes \DR {d-r} {X|D_{\underline{n}+\underline{1}}} \to \DR {d} {X|\D 1}=\DR d X,\; (\xi,\eta) \mapsto \xi\wedge\eta  ; \label{twisted.pairing-1}\\
		Z\DR r {X|D_{-\underline{n}}} &\otimes \DR {d-r} {X|D_{\underline{n}+\underline{1}}} /d\DR {d-r-1} {X|D_{\underline{n}+\underline{1}}} \to \DR {d} {X|\D 1}=\DR d X,\; (\xi, \eta) \mapsto -C(\xi \wedge \eta); \label{twisted.pairing-2}
		\end{align}
		are perfect. In fact the perfectness of the pairing (\ref{twisted.pairing-2}) follows from \cite[Lem. 1.7]{milnesurface}.

		By Grothendieck-Serre duality, we obtain the following isomorphisms as $k$-vector spaces,
		\[ H^i(X, \DR r {X|D_{-\underline n}}) \cong H^{d-i}(X, \DR {d-r}{X|D_{\underline n+\underline 1}})^* ,\] and \[H^i(X, Z\DR r {X|D_{-\underline n}}) \cong H^{d-i}(X, \DR {d-r}{X|D_{\underline n+\underline 1}}/d\DR {d-r-1} {X|D_{\underline n+\underline 1}})^*. \]
		Note that, for any two $k$-vector spaces $V$ and $W$, an isomorphism of $k$-vector spaces \[ W\cong\Hom_k(V,k)=:V^\ast\] uniquely corresponds to an isomorphism of $\Fp$-vector spaces \[ W\cong\Hom_{\Fp}(V,\Fp)=:V^{\vee}.\]
		The above two isomorphisms give the isomorphism $(1)$ in the following commutative diagram:
		\[ \xymatrix{
			\varprojlim\limits_{n}\mathbb{H}^{d+1-i}(X,W_1\mathscr{G}^{d-r,\bullet}_{ \un+\ul 1}) \ar[d]_-{\cong}^-{(2)} \ar[r]^-{\cong}_-{(1)} &(\varinjlim\limits_{n}\mathbb{H}^i(X, W_1\mathscr{F}^{r,\bullet}_{-\un}))^{\vee}\ar@{=}[d]^-{(3)}\\
			\varprojlim\limits_{\underline{n}}H^{d+1-i}(X,\DRlog {d-r}{X|D_{\underline n +\underline 1}})\ar@{=}[d]  &  (\mathbb{H}^i(X, [j_*Z\DR r U \xrightarrow{1-C} j_*\DR r U]))^{\vee}\\
			\varprojlim\limits_{D}H^{d+1-i}(X,\DRlog {d-r}{X|D}) \ar[r]  & H^{i}(U,\DRlog r U)^{\vee}\ar[u]^{\cong}_{(4)},
		}\]
		where the isomorphism (2) is induced by Theorem \ref{logtocoherent},  (3) follows from the observation that  $j_*\DR r U =\varinjlim_{\underline n}\DR r {X|D_{-\underline n}}$, and the isomorphism (4) is due to the fact that $Rj_*\Omega_{U,\log}^r \cong [j_*Z\DR r U \xrightarrow{1-C} j_*\DR r U]$. Therefore the last horizontal map is an isomorphism.

	\end{proof}

	In particular, for $i=1$ and $r= 0$ we get isomorphisms
	\begin{equation*}
	\varprojlim\limits_{D} H^d(X,\DRWlog m d {X|D})\xrightarrow{\cong} H^1(U,\Z/p^m\Z)^\vee \cong \pi^{ab}_1(U)/p^m,
	\end{equation*}
	and 
	\begin{equation*}
	H^1(U,\Z/p^m\Z) \xrightarrow{\cong} \varinjlim\limits_{D} H^d(X,\DRWlog m d {X|D})^{\vee}
	\end{equation*}
	
	These isomorphisms can be used to define a measure of ramification for  \'{e}tale abelian covers of $U$ whose degree divides $p^m$.
	\begin{definition}\label{new.filtration}
		For any $D\in \Sigma$, we define
		\[  \Fil_DH^1(U,\Z/p^m\Z):= H^d(X,\DRWlog m d {X|D})^{\vee},\]
		\[  \Fil_DH^1(U,\Q/\Z): =H^1(U,\Q/\Z)\{p'\} \bigoplus \bigcup_{m\geq 1} \Fil_DH^1(U,\Z/p^m\Z), \]
		where $H^1(U,\Q/\Z)\{p'\}$ is the prime-to-$p$ part of $H^1(U,\Q/\Z)$.
		Dually we define
		\[   \pi^{\text{ab}}_1(X,D)/p^m :=\Hom(\Fil_DH^1(U,\Z/p^m\Z),\Z/p^m\Z),  \]	
		\[ \pi^{\text{ab}}_1(X,D):=\Hom(\Fil_DH^1(U,\Q/\Z),\Q/\Z). \]
		
	\end{definition}
	The group $ \pi^{\text{ab}}_1(X,D)/p^m$ is a quotient of $\pi_1^{\text{ab}}(U)/p^m$, which can be thought of as classifying abelian \'etale coverings of $U$ whose degree divides $p^m$ with ramification bounded by $D$. These groups are important objects in higher-dimensional class field theory.

	\bigskip
	
	\section{Duality over perfect fields}
	\bigskip
	When the base field $k$ is finite, our duality theory is formulated by 
	endowing the cohomology groups with the structure of topological groups. When the base field $k$ is not finite, it is necessary to endow the cohomology groups with stronger structures, namely the structures of pro-algebraic and ind-algebraic groups, and use Breen-Serre duality instead of Pontryagin duality. 
	In this section, $k$ denotes a perfect field of characteristic $p>0$, not necessarily finite, and we put $S=\Spec(k)$.

	\subsection{The relative perfect \'etale site}\label{relative.perfect.site}
	
	Recall a scheme $T$ is said to be \emph {perfect} if the absolute Frobenius $F: T\to T$ is an isomorphism.  
	For any $S$-scheme $X$, the \emph{perfection} $X^{\text{pf}}$ of $X$ is the projective limit of the system
	\[  X_{\text{red}}\xleftarrow{F} X^{(p^{-1})}_{\text{red}} \xleftarrow{F}\cdots \xleftarrow{F} X^{(p^{-n})}_{\text{red}} \xleftarrow{F}\cdots, \]
	where $X^{(p^{-n})}_{\text{red}} $ is the scheme $X_{\text{red}}$ with the structure map $F^n\circ \pi: X \to S$. It is a perfect scheme, and has the universal property that 
	\[ \Hom_S(X,Y)=\Hom_S(X^{\text{pf}}, Y) \]
	for any perfect $S$-scheme $Y$. 
	A perfect $S$-scheme $X$ is said to be \emph{algebraic} if it is the perfection of a scheme of finite type over $S$. 
	One sees easily that the perfect algebraic group schemes over $S$ form an abelian category. Let $(\Pf/S)_{\text{\'et}}$  be the \emph{perfect \'etale site} over $S$, i.e., the category of perfect schemes over $S$ with \'etale topology. 
	\medbreak
	
	In what follows we fix a smooth proper morphism $\pi: X \to S$ and an effective divisor $D$ such that Supp($D$) is a simple normal crossing divisor on $X$.
	Let $j: U:=X-D \hookrightarrow X$ be the complement of $D$. 
	Let $(\Pf X/S)_{\text{\'et}}$  be the \emph{relative perfect \'etale site} over $X/S$, i.e., the category of pairs $(T,Y)$, where $T$ is a perfect scheme over $S$ and $Y$ is \'etale over $X\times_ST$ equipped with \'etale topology. 
	We define $\cX$ and $\cS$ to be the category of abelian sheaves on $(\Pf X/S)_{\text{\'et}}$ and on $(\Pf/S)_{\text{\'et}}$, respectively. For any integer $m\geq 1$, we denote $\cX(p^m)$ (resp. $\cS(p^m)$) to be the category of sheaves of $\Z/p^m\Z$-modules on $(\Pf X/S)_{\text{\'et}} $ (resp. $(\Pf/S)_{\text{\'et}}$).
	The structure morphism $\pi: X\to S$ induces a morphism of sites
	\[ \pi: (\Pf X/S)_{\text{\'et}} \to (\Pf/S)_{\text{\'et}},\;  (T,Y) \mapsto T,\]
	which gives rise to adjoint functors
	\[ \xymatrix{ \pi_* : \cX \ar@<-0.5ex>[r]  &\cS : \pi^* \ar@<-0.5ex>[l]      &\text{and}  & \pi_* : \cX(p^m) \ar@<-0.5ex>[r]  &\cS(p^m) : \pi^*. \ar@<-0.5ex>[l]  }\]
	
	Definition \ref{relative.definition} gives an object $\DRWlog{m}{r}{X|D}$ of 
	$\cX(p^m)$ such that $R^i\pi_*\DRWlog{m}{r}{X|D}$ is the sheaf on $(\Pf/S)_{\text{\'et}}$ associated to the presheaf 
	\[T \mapsto H^i(X_T, \DRWlog m r {X_T|D_T})\quad (T\in Ob((\Pf/S)_{\text{\'et}})).\]

	\subsection{Duality theorem}

	By \eqref{finite.WFn}  we have an isomorphism
	\[ \varinjlim_{\un}W_m\mathscr{F}^{r,\bullet}_{-\un}=[j_*Z_1\DRW m r U \xrightarrow{1-C} j_*\DRW m r U] \cong  Rj_*\DRWlog m r U,\]
	where the second isomorphism follows from the fact that $j$ is affine. Therefore
	\begin{equation}\label{Fn=Ulog} 
	R\varinjlim_{\un} R\pi_*W_m\mathscr{F}^{r,\bullet}_{-\un}= R\pi_*Rj_*\DRWlog m r U \in D^b(\cS(p^m)), 
	\end{equation}
	since $R\varinjlim_{\un}$ commutes with $R\pi_*$.
	By \eqref{Wm-2-pairing} and \eqref{finite.WHn-log}, we have a map 
	\begin{equation*}
	R\pi_*W_m\mathcal{G}^{d-r,\bullet}_{\un+\ul{1}} \to R\sHom_{\cS}(R\pi_*W_m\mathcal{F}^{r,\bullet}_{-\un}, R\pi_*\DRWlog m d X).
	\end{equation*}
	By taking the homotopy limit $R\varprojlim_{\un}$ on both sides, we obtain a map
	\[\begin{aligned} 
	R\varprojlim_{\un}R\pi_*W_m\mathscr{G}_{\un+\ul{1}}^{d-r, \bullet}  \to  &R\varprojlim_{\un}R\sHom_{\cS}(R\pi_*W_m\mathscr{F}^{r,\bullet}_{-\un}, R\pi_*\DRWlog m d X)\\
	&\simeq R\sHom_{\cS}(R\varinjlim_{\un}R\pi_*W_m\mathscr{F}^{r,\bullet}_{-\un}, R\pi_*\DRWlog m d X) \\
	&\simeq R\sHom_{\cS}(R\pi_* Rj_* \DRWlog m r U, R\pi_*\DRWlog m d X) \\
	&\to R\sHom_{\cS}(R\pi_* Rj_* \DRWlog m r U, \Z/p^m\Z)[-d] 
	\end{aligned}\]
	where the second isomorphism follows from \eqref{Fn=Ulog} and the last map is induced by the trace map $\text{Tr}: R\pi_*\DRWlog m d X \to\Z/p^m\Z[-d]$.
	Thus Theorem \ref{exactproobj} gives rise to a map 
	\begin{equation}\label{perfect.duality}
	R\varprojlim\limits_{D} R\pi_*\DRWlog m {d-r} {X|D} \to R\sHom_{\cS(p^m)}( R\pi_*Rj_* \DRWlog m r U , \Z/p^m\Z)[-d].
	\end{equation}

	\begin{theorem}\label{duality.thm.perfect}
		The map (\ref{perfect.duality}) is an isomorphism in $D^b(\cS(p^m))$.
	\end{theorem}
	\begin{proof}\renewcommand{\qedsymbol}{}
		By the same method as in the proof of Theorem  \ref{duality.thm.finite}, 
		we reduce the claim to the case $m=1$. We then use the following result
		from \cite[Prop. 2.1]{milnesurface}, \cite[Lem. 3.6]{berthelotmilne}.
	\end{proof}
	
	\begin{proposition}
		Let $\cL$ be a locally free $\cO_X$-module of finite rank and put $\cL^{\vee}=\sHom_{\cO_X}(\cL, \cO_X)$. Then the natural pairing 
		\[ \cL \times (\cL^{\vee}\otimes \DR d X)  \to [\DR d X \xrightarrow{1-C} \DR d X][1] \cong \DRlog d X [1] \]
		and the trace map $R\pi_*\DRlog d X \to \Z/p\Z[-d]$ induces an isomorphism
		\[ R\pi_*\cL \xrightarrow{\cong} R\sHom_{\cS(p)}(R\pi_*(\cL^{\vee}\otimes\DR d X), \Z/p\Z)[-d+1]. \]
	\end{proposition}
	\begin{corollary}
		The perfect pairings (\ref{twisted.pairing-1}) and (\ref{twisted.pairing-2}) induces isomorphisms
		\[ R\pi_*\DR {d-r} {X|D_{\underline{n}+\underline{1}}} \xrightarrow{\cong} R\sHom_{\cS(p)}(R\pi_*\DR r {X|D_{-\underline{n}}}, \Z/p\Z)[-d+1]; \]
		\[ R\pi_*\DR {d-r} {X|D_{\underline{n}+\underline{1}}} /d\DR {d-r-1} {X|D_{\underline{n}+\underline{1}}} \xrightarrow{\cong} R\sHom_{\cS(p)}(R\pi_*Z\DR r {X|D_{-\underline{n}}}, \Z/p\Z)[-d+1].\]
		Therefore, we have an isomorphism
		\[  R\pi_*W_1\mathscr{G}^{\bullet}_{ \un+\ul{1}}   \xrightarrow{\cong} R\sHom_{\cS(p)}(R\pi_*W_1\mathscr{F}^{\bullet}_{-\un} ,\Z/p\Z)[-d],  \]
		where $W_1\mathscr{F}^{r,\bullet}_{-\un} $ and $W_1\mathscr{G}^{d-r,\bullet}_{\un+\ul 1}$ were defined in (\ref{finite.WFn}) and (\ref{finite.WGn}).
	\end{corollary}

	\begin{proof}[Proof of Theorem \ref{duality.thm.perfect}(continued)] By taking the limit, we obtain
		\begin{align*}
		R\varprojlim\limits_{D}R\pi_*\DRlog {d-r} {X|D} &\xrightarrow{\cong} R\varprojlim\limits_{\underline n}R\pi_*\DRlog {d-r} {X|D_{\underline{n}+\underline{1}}} \xrightarrow{\cong}R\varprojlim\limits_{\ul n}R\pi_*W_1\mathscr{G}^{d-r,\bullet}_{\un+\ul 1}\\
		&\xrightarrow{\cong}R\varprojlim\limits_{\ul n} R\sHom_{\cS(p)}(R\pi_*W_1\mathscr{F}^{r,\bullet}_{-\un},\Z/p\Z)[-d]\\
		&\xrightarrow{\cong}R\sHom_{\cS(p)}(R\varinjlim_{\underline n} R\pi_*W_1\mathscr{F}^{r,\bullet}_{-\un},\Z/p\Z)[-d]\\
		&\xrightarrow{\cong}R\sHom_{\cS(p)}( R\pi_*Rj_*\DRlog r U,\Z/p\Z)[-d]
		\end{align*}
		
		This is our theorem in the case that $m=1$.
	\end{proof}
	
	\begin{remark}
		In fact we can endow $R\varprojlim\limits_{D} R\pi_*\DRWlog m {d-r} {X|D}$ with a structure of a complex of proalgebraic groups, i.e.,  as an object in the bounded derived category of quasi-unipotent proalgebraic groups,   and similarly view $R\pi_*Rj_*\DRWlog m r U $ as an object in the bounded derived category of quasi-unipotent indalgebraic groups. Then Theorem \ref{perfect.duality} identifies $R\varprojlim\limits_{D} R\pi_*\DRWlog m {d-r} {X|D}$ with the Breen-Serre dual of $R\pi_*Rj_*\DRWlog m r U $(cf. \cite[\S 2.5]{pepin}).
	\end{remark} 
	
	\bigskip

	\newcommand{\etalchar}[1]{$^{#1}$}
	\providecommand{\bysame}{\leavevmode\hbox to3em{\hrulefill}\thinspace}
	\providecommand{\MR}{\relax\ifhmode\unskip\space\fi MR }
	\providecommand{\MRhref}[2]{%
		\href{http://www.ams.org/mathscinet-getitem?mr=#1}{#2}
	}
	\providecommand{\href}[2]{#2}


\end{document}